\documentclass[10pt,reqno]{amsart}
\usepackage{amsmath,amstext,amssymb,amsopn,amsthm,mathrsfs}
\usepackage{epsfig,graphics,latexsym,graphicx,color}
\allowdisplaybreaks

\newtheorem{theorem}{Theorem}[section]
\newtheorem{lemma}[theorem]{Lemma}

\newtheorem{proposition}[theorem]{Proposition}

\newtheorem{remark}[theorem]{Remark}       

\numberwithin{equation}{section}

\begin{document}
\title[]{Characterizations of Lorentz Type Sobolev Multiplier Spaces and Their Preduals}
\date{}
\author[Keng Hao Ooi]
{Keng Hao Ooi}
\address{Department of Mathematics,
National Central University,
No.300, Jhongda Rd., Jhongli City, Taoyuan County 32001, Taiwan (R.O.C.).}
\email{kooi1@math.ncu.edu.tw}

\begin{abstract}
We provide several characterizations of Sobolev multiplier spaces of Lorentz type and their preduals. Block decomposition and K\"othe dual of such preduals are discussed. As an application, the boundedness of local Hardy-Littlewood maximal function on these spaces will be justified.
\end{abstract}

\maketitle

\section{Introduction}
Let $n\in\mathbb{N}$, $0<\alpha<\infty$, and $1<s\leq n/\alpha$. The Bessel capacities ${\rm Cap}_{\alpha,s}(\cdot)$ are defined by
\begin{align*}
\text{Cap}_{\alpha,s}(E)&=\inf\{\|f\|_{L^{s}(\mathbb{R}^{n})}^{s}:f\geq 0,~G_{\alpha}\ast f\geq 1~\text{on}~E\},\quad E\subseteq\mathbb{R}^{n},
\end{align*}
where $G_{\alpha}(x)=\mathcal{F}^{-1}[(1+\left|\cdot\right|^{2})^{-\alpha/2}](x)$, $x\in\mathbb{R}^{n}$ are the Bessel kernels, and $\mathcal{F}^{-1}$ is the inverse distributional Fourier transform on $\mathbb{R}^{n}$. The Sobolev multiplier spaces $M_{p}^{\alpha,s}$ are studied in \cite{OP} with the following form that 
\begin{align*}
\|f\|_{M_{p}^{\alpha,s}}=\sup_{K}\left(\frac{\int_{K}|f(x)|^{p}dx}{{\rm Cap}_{\alpha,s}(K)}\right)^{\frac{1}{p}},\quad 1<p<\infty,
\end{align*} 
where the supremum is taken over all compact sets $K\subseteq\mathbb{R}^{n}$ with nonzero capacities. As noted in \cite{OP}, the spaces $M_{p}^{\alpha,s}$ appear in many super-critical nonlinear PDEs including the Navier-Stokes system, and a function $f\in M_{p}^{\alpha,s}$ if and only if $|f|^{p/s}$ belongs to the space of Sobolev multipliers $M(H^{\alpha,s}\rightarrow L^{s})$ that has been studied in \cite{MS}. Several characterizations of preduals of $M_{p}^{\alpha,s}$ are also studied in \cite{OP}.

In the present paper, we generalize $M_{p}^{\alpha,s}$ to Lorentz type spaces and characterize their preduals. First we introduce the fundamental terminologies. In the sequel, we abbreviate $\mathcal{R}$, $\mathcal{C}$ for $\mathbb{R}^{n}$,  ${\rm Cap}_{\alpha,s}(\cdot)$ respectively. Given a positive measure $\mu$ on $\mathcal{R}$, let $\mathcal{M}(\mathcal{R},\mu)$ be the set of all extended real-valued $\mu$-measurable functions on $\mathcal{R}$. A functional $\left\|\cdot\right\|:\mathcal{M}(\mathcal{R},\mu)\rightarrow[0,\infty]$ is said to be a quasi-norm if there is a positive constant $\kappa_{0}$ such that 
\begin{align*}
\|f+g\|&\leq\kappa_{0}(\|f\|+\|g\|),\\
\|\lambda f\|&=|\lambda|\|f\|
\end{align*}
for $f,g\in\mathcal{M}(\mathcal{R},\mu)$, $\lambda\in\mathbb{R}$, and $\|f\|=0$ exactly when $f=0$ $\mu$-almost-everywhere. Subsequently, assuming that $X,Y$ are subspaces of $\mathcal{M}(\mathcal{R},\mu)$ with the corresponding quasi-norms $\left\|\cdot\right\|_{X},\left\|\cdot\right\|_{Y}$, we write $X\hookrightarrow Y$ provided that $\left\|\cdot\right\|_{Y}\leq c\left\|\cdot\right\|_{X}$ for a positive constant $c$, and the symbol $X\approx Y$ refers to both $X\hookrightarrow Y$ and $Y\hookrightarrow X$. Since $\left\|\cdot\right\|_{X}$ can induce a topological vector space for $X$, then $X$ is isomorphic to $Y$ exactly when $X\approx Y$, the isomorphism becomes isometric when $\left\|\cdot\right\|_{X}=\left\|\cdot\right\|_{Y}$. Let $X^{\ast}$ be the space of all continuous linear functionals on $X$. We write $X^{\ast}\approx Y$ provided that for any $\mathcal{L}\in X^{\ast}$, there is a unique $g\in Y$ such that 
\begin{align*}
\mathcal{L}(f)=\int_{\mathcal{R}}f(x)g(x)d\mu(x),\quad f\in X,
\end{align*} 
and for each $g\in Y$, the canonical linear functional $\mathcal{L}_{g}$ induced by $g$ satisfies the above integral representation with $\mathcal{L}_{g}$ in place of $\mathcal{L}$, in which case, the operator norm $\|\mathcal{L}\|$ satisfies $c^{-1}\|g\|_{Y}\leq\|\mathcal{L}\|\leq c\|g\|_{Y}$ for some positive constant $c$. We also call $X$ a predual of $Y$ for $X^{\ast}\approx Y$. On the other hand, we define the K\"othe dual $X'$ of $X$ to be the set of all functions $f\in\mathcal{M}(\mathcal{R},\mu)$ such that 
\begin{align*}
\|f\|_{X'}=\sup\left\{\int_{\mathcal{R}}|f(x)g(x)|d\mu(x):\|g\|_{X}\leq 1\right\}
\end{align*}
are finite. It is evident that 
\begin{align*}
\int_{\mathcal{R}}|f(x)g(x)|d\mu(x)\leq\|f\|_{X}\|g\|_{X'},
\end{align*}
and hence $X\hookrightarrow X''=(X')'$.

Given $0<p,q<\infty$, the Lorentz space $L^{p,q}(\mu)$ is defined to be the set of all functions $f\in\mathcal{M}(\mathcal{R},\mu)$ such that
\begin{align*}
\|f\|_{L^{p,q}(\mu)}=p^{\frac{1}{q}}\left(\int_{0}^{\infty}\mu(\{x\in\mathcal{R}:|f(x)|>t\})^{\frac{q}{p}}t^{q}\frac{dt}{t}\right)^{\frac{1}{q}}
\end{align*}
are finite. We note that $\left\|\cdot\right\|_{L^{p,q}(\mu)}$ is a quasi-norm by a standard exercise in analysis. When $\mu$ is the usual Lebesgue measure on $\mathcal{R}$, we reduce the symbols $\mathcal{M}(\mathcal{R},\mu)$ and $L^{p,q}(\mu)$ to $\mathcal{M}(\mathcal{R})$ and $L^{p,q}$ respectively. Note that $L^{p}=L^{p,p}$ by an application of Fubini's theorem. To extend $M_{p}^{\alpha,s}$ into Lorentz type, we have two choices. Let $1<p<\infty$ and $1<q<\infty$. We define
\begin{align*}
\|f\|_{\mathcal{M}^{p,q}}=\sup_{K}\frac{\|f\chi_{K}\|_{L^{p,q}}}{\mathcal{C}(K)^{1/p}}
\end{align*}
and
\begin{align*}
\|f\|_{M^{p,q}}=\sup_{K}\frac{\|f\chi_{K}\|_{L^{p,q}}}{\mathcal{C}(K)^{1/q}},
\end{align*}
and the symbols $\mathcal{M}^{p,q},M^{p,q}$ designate the set of all functions $f\in\mathcal{M}(\mathcal{R})$ with finite $\|f\|_{\mathcal{M}^{p,q}},\|f\|_{M^{p,q}}$ respectively. Then $\mathcal{M}^{p,p}=M^{p,p}=M_{p}^{\alpha,s}$. Further, for any $1<p<\infty$ and $1<r<q<\infty$, it holds that $L^{p,r}\hookrightarrow L^{p,q}$ (see \cite[Proposition 1.4.10]{GL}), then we deduce that 
\begin{align*}
\mathcal{M}^{p,r}\hookrightarrow\mathcal{M}^{p,q}.
\end{align*}
Among the two candidates for a Lorentz type extension of $M_{p}^{\alpha,s}$, namely, $\mathcal{M}^{p,q}$ and $M^{p,q}$, the space $\mathcal{M}^{p,q}$ appears to be the more natural choice. Nevertheless, it will be seen in the following context that the preduals of $\mathcal{M}^{p,q}$ are more restrictive than of $M^{p,q}$ (see Theorems \ref{embed} and \ref{dual N} below). Further, we will prove that if $1<r<q\leq p<\infty$, $\alpha s=n$ (certain range of $p,q$ is assumed when $\alpha s<n$), then
\begin{align*}
M^{p,r}&\hookrightarrow M^{p,q},\\
L^{\infty}&\hookrightarrow M^{p,q}
\end{align*}
hold by an argument of Strichartz localization principle (see Proposition \ref{addition} below). 

Our first result reads as the following. Essentially it says that the K\"othe dual of $M^{p,q}$ is its predual.
\begin{theorem}\label{first predual}
Let $1<p<\infty$ and $1<q<\infty$. Then the following isomorphisms
\begin{align*}
[(M^{p,q})']^{\ast}&\approx M^{p,q},\\
[(\mathcal{M}^{p,q})']^{\ast}&\approx\mathcal{M}^{p,q}
\end{align*}
hold. The isomorphisms $\approx$ become isometric for $p=q$.
\end{theorem}

It will be shown also that $(M^{p,q})''\approx M^{p,q}$, similar assertion also applies to $\mathcal{M}^{p,q}$. Another preduals of $M^{p,q}$ are given by block type spaces, which will be defined as follows. Let $1<p<\infty$ and $1<q<\infty$. A function $b\in\mathcal{M}(\mathcal{R})$ is called a $B^{p,q}$ block if 
\begin{enumerate}
\item $\{x\in\mathcal{R}:b(x)\ne 0\}\subseteq E$ for some bounded set $E\subseteq\mathcal{R}$.

\medskip
\item $\mathcal{C}(E)^{1/q'}\|b\|_{L^{p,q}}\leq 1$.
\end{enumerate}
Whereas a $\mathcal{B}^{p,q}$ block is defined as in above except that $\mathcal{C}(E)^{1/p'}\|b\|_{L^{p,q}}\leq 1$. Then the space $B^{p,q}$ is defined to be the set of all functions $f\in\mathcal{M}(\mathcal{R})$ that having the following block decomposition
\begin{align}\label{definition of block decomposition}
f=\sum_{k}\lambda_{k}b_{k},
\end{align}
where the convergence of the series is pointwise almost everywhere, $\{\lambda_{k}\}\in\ell^{1}$, and each $b_{k}$ is a $B^{p,q}$ block. Further, we equip $B^{p,q}$ with
\begin{align}\label{definition of block norm}
\|f\|_{B^{p,q}}=\inf\left\{\sum_{k}|\lambda_{k}|:f=\sum_{k}\lambda_{k}b_{k}\right\},
\end{align}
where the infimum is taken over all such block decompositions. The spaces $\mathcal{B}^{p,q}$ are defined similarly with the corresponding $\mathcal{B}^{p,q}$ blocks. In the next theorem, we denote the H\"older's conjugate of $p$ by $p'=p/(p-1)$, where $1<p<\infty$. 
\begin{theorem}\label{second predual}
Let $1<p<\infty$ and $1<q<\infty$. Then the following isomorphisms
\begin{align*}
&(B^{p',q'})^{\ast}\approx M^{p,q}\approx(B^{p',q'})',\\
&(\mathcal{B}^{p',q'})^{\ast}\approx\mathcal{M}^{p,q}\approx(\mathcal{B}^{p',q'})'
\end{align*}
hold. Moreover, both $B^{p',q'}$ and $\mathcal{B}^{p',q'}$ are Banach spaces.
\end{theorem}

The block spaces $B_{p}^{\alpha,s}$, which correspond to $M_{p}^{\alpha,s}$, are also introduced in \cite{OP}, and one may infer easily that $B^{p,q}$ generalizes $B_{p}^{\alpha,s}$. Moreover, as suggested in \cite{OP}, there are ``N-type" spaces that correspond to block spaces, which we introduce here. To this end, for any nonnegative function $\omega$ on $\mathcal{R}$, we denote $\|\omega\|_{L^{1}(\mathcal{C})}$ by
\begin{align*}
\|\omega\|_{L^{1}(\mathcal{C})}=\int_{0}^{\infty}\mathcal{C}(\{x\in\mathcal{R}:\omega(x)>t\})dt.
\end{align*}
The space $L^{1}(\mathcal{C})$ consists of those $\omega$ with finite $\|\omega\|_{L^{1}(\mathcal{C})}$. Further, we say that a nonnegative measurable function $\omega$ is a weight if $\omega$ is locally integrable on $\mathcal{R}$ with $\omega>0$ almost everywhere. Subsequently, we define the local Hardy-Littlewood maximal function ${\bf M}^{\rm loc}$ by
\begin{align*}
{\bf M}^{\rm loc}f(x)=\sup_{0<r\leq 1}\frac{1}{|B_{r}(x)|}\int_{B_{r}(x)}|f(y)|dy,\quad x\in\mathbb{R}^{n},
\end{align*}
where $f\in\mathcal{M}(\mathcal{R})$ is locally integrable on $\mathcal{R}$. Then the local Muckenhoupt $A_{1}$ class (see \cite{RV} and \cite{OK4}) consists of all weights $\omega$ such that 
\begin{align*}
{\bf M}^{\rm loc}\omega\leq C\omega
\end{align*}
almost everywhere, where $C$ is a positive constant. The infimum of all such $C$ is denoted by $[\omega]_{A_{1}^{\rm loc}}$. Let $1<p<\infty$ and $1<q<\infty$. The $N^{p,q}$ space is defined to be the set of all functions $f\in\mathcal{M}(\mathcal{R})$ such that 
\begin{align*}
\|f\|_{N^{p,q}}=\inf\{\|f\omega^{-1/q'}\|_{L^{p,q}}:\omega\in L^{1}(\mathcal{C})\cap A_{1}^{\rm loc},\|\omega\|_{L^{1}(\mathcal{C})}\leq 1,[\omega]_{A_{1}^{\rm loc}}\leq\mathfrak{c}\}
\end{align*}
are finite, where $\mathfrak{c}$ is a fixed positive constant depending only on $n,\alpha,s$ defined in (\ref{universal}) below. The connection of $B^{p,q}$ with $N^{p,q}$ is given as follows.
\begin{theorem}\label{embed}
Let $1<p<\infty$ and $1<q<\infty$. If $1<q\leq p<\infty$, then 
\begin{align}\label{first embed}
B^{p,q}\hookrightarrow N^{p,q}.
\end{align}
If $1<p\leq q<\infty$, then
\begin{align}\label{second embed}
N^{p,q}\hookrightarrow B^{p,q}.
\end{align}
\end{theorem}
In contrast, one is tempted to define the ``N-type" spaces that correspond to $\mathcal{M}^{p,q}$ by 
\begin{align*}
\inf_{\omega}\|f\omega^{-1/p'}\|_{L^{p,q}},
\end{align*}
where $\omega$ are those admissible weights. However, it will be seen from the proof of Theorem \ref{embed} that these ``N-type" spaces are neither comparable to $\mathcal{B}^{p,q}$ nor preduals of $\mathcal{M}^{p,q}$. The ``N-type" spaces $N_{p}^{\alpha,s}$ introduced in \cite{OP} are in fact Banach function spaces (see section $3$ below for the definition). One of the main properties of Banach function spaces is the Fatou's property, i.e., for any functions $0\leq f_{1}\leq f_{2}\leq\cdots$, the function $f=\sup_{n\in\mathbb{N}}f_{n}$ satisfies 
\begin{align}\label{fatou to describe}
\|f\|=\sup_{n\in\mathbb{N}}\|f_{n}\|,
\end{align}
where $\left\|\cdot\right\|$ is the norm induced by the Banach function space. Given a quasi-norm $\left\|\cdot\right\|_{X}$ (here $X$ is not necessary a Banach function space), if there exists a positive constant $\kappa$ such that 
\begin{align*}
\|f\|_{X}\leq\kappa\cdot\sup_{n\in\mathbb{N}}\|f_{n}\|_{X},
\end{align*}
where $f,f_{n}$ are arbitrary functions as in (\ref{fatou to describe}), then we say $X$ has the weak Fatou's property. Below we prove that $N^{p,q}$ possesses the weak Fatou's property for certain range of $p,q$. Besides, it will be shown that $N^{p,q}$ is normable for some $p,q$, i.e., $N^{p,q}\approx X$ for some normed space $X$.
\begin{proposition}\label{fatou property}
Let $1<p<\infty$ and $1<q<\infty$. Then $N^{p,q}$ is a quasi-norm space that satisfying the weak Fatou's property. If $1<q\leq p<\infty$, then $N^{p,q}$ is a complete normable space.
\end{proposition}

Denote by $C_{0}$ the space of all compactly supported continuous functions on $\mathcal{R}$. Then the density of $C_{0}$ in $B^{p,q},N^{p,q}$ is addressed in the following proposition.
\begin{proposition}\label{density}
The space $C_{0}$ is dense in both $B^{p,q}$ and $N^{p,q}$ for $1<p<\infty$ and $1<q<\infty$.
\end{proposition}

As in \cite{OP}, we will prove that $N^{p,q}$ is a predual of $M^{p,q}$. To this end, we will characterize $M^{p,q}$ in terms of $A_{1}^{\rm loc}$ weights, and use this characterization to prove the duality result. 
\begin{theorem}\label{char M}
Let $1<q\leq p<\infty$ and
\begin{align}\label{M char}
M=\sup\{\|f\omega^{\frac{1}{q}}\|_{L^{p,q}}:\omega\in L^{1}(\mathcal{C})\cap A_{1}^{\rm loc},\|\omega\|_{L^{1}(\mathcal{C})}\leq 1,[\omega]_{A_{1}^{\rm loc}}\leq\mathfrak{c}\}.
\end{align}
Then 
\begin{align*}
c^{-1}\|f\|_{M^{p,q}}\leq M\leq c\|f\|_{M^{p,q}},
\end{align*}
where $c$ is a positive constant depending only on $n,p,q$ and the parameters $\alpha,s$ in $\mathcal{C}={\rm Cap}_{\alpha,s}(\cdot)$.
\end{theorem}

As promised, the duality of $N^{p,q}$ is given as follows.
\begin{theorem}\label{dual N}
Let $1<p<\infty$ and $1<q<\infty$. Then it holds that 
\begin{align}\label{first predual N}
(N^{p',q'})^{\ast}\hookrightarrow M^{p,q}.
\end{align}
If $1<q\leq p<\infty$, then the following isomorphisms
\begin{align}\label{predual}
(N^{p',q'})^{\ast}\approx M^{p,q}\approx(N^{p',q'})'
\end{align}
hold and 
\begin{align}\label{chain}
N^{p',q'}\hookrightarrow B^{p',q'}\hookrightarrow (M^{p,q})'.
\end{align}
\end{theorem}
The question of whether $\approx$ holds in (\ref{chain}) for $q < p$ remains open, primarily because the Fatou's property of $N^{p,q}$ has not been confirmed (see the proof of \cite[Theorem 5.3]{OP}). Note that Proposition \ref{fatou property} only establishes the weak Fatou's property of $N^{p,q}$.

The boundedness of local maximal function ${\bf M}^{\rm loc}$ on $M^{p,q}$ and the preduals $N^{p,q}$ are addressed below.
\begin{theorem}\label{maximal}
Let $1<p<\infty$ and $1<q<\infty$. Then there exist positive constants $\gamma,\overline{\gamma}$ depending only on $n$ and the parameters $\alpha,s$ in $\mathcal{C}={\rm Cap}_{\alpha,s}(\cdot)$ such that 
\begin{align}
{\bf M}^{\rm loc}&:M^{p,q}\rightarrow M^{p,q},\quad 1<q\leq p<(1+\gamma)q,\label{boundedness 1}\\
{\bf M}^{\rm loc}&:N^{p,q}\rightarrow N^{p,q},\quad 1<\frac{q+\overline{\gamma}}{1+\overline{\gamma}}<p\leq q\label{boundedness 2}
\end{align}
are bounded.
\end{theorem}

Now we consider the weak type Sobolev multiplier spaces. Let $\mathcal{M}^{p,\infty}$ be the subspace of $\mathcal{M}(\mathcal{R})$ such that 
\begin{align*}
\|f\|_{\mathcal{M}^{p,\infty}}=\sup_{K}\frac{\|f\chi_{K}\|_{L^{p,\infty}}}{\mathcal{C}(K)^{1/p}}
\end{align*}
is finite for any $f\in M^{p,\infty}$. We note that  
\begin{align*}
\|f\|_{\mathcal{M}^{p,\infty}}=\sup_{K}\sup_{t>0}\frac{t|\{|f|\chi_{K}>t\}|^{1/p}}{\mathcal{C}(K)^{1/p}}=\sup_{t>0}t\|\chi_{\{|f|>t\}}\|_{M_{p}^{\alpha,s}},
\end{align*}
from which $\mathcal{M}^{p,\infty}$ can be viewed as the weak Sobolev multiplier spaces with respect to $M_{p}^{\alpha,s}$. The preduals of $\mathcal{M}^{p,\infty}$ are given by block spaces $\mathcal{B}^{p',q}$, where $1<p<\infty$ and $0<q\leq 1$. While the blocks $b$ of $\mathcal{B}^{p',q}$ satisfy $\mathcal{C}(E)^{1/p'}\|b\|_{L^{p,q}}\leq 1$ with bounded support $E\subseteq\mathcal{R}$, and we endow $\mathcal{B}^{p',q}$ with the block decompositions (\ref{definition of block decomposition}) and (\ref{definition of block norm}).
\begin{theorem}\label{weak dual}
Let $1<p<\infty$ and $0<q\leq 1$. Then the following isomorphism 
\begin{align*}
(\mathcal{B}^{p',q})^{\ast}\approx\mathcal{M}^{p,\infty}
\end{align*}
holds.
\end{theorem}

Note that when $p=q$, the term $\|f\chi_{K}\|_{L^{p,q}}^{q}$ can be viewed as a Radon-Nikodym measure that $d\mu=|f|^{p}\chi_{K}dx$. This leads to define that
\begin{align*}
\|\mu\|_{\mathfrak{M}}=\sup_{K}\frac{|\mu|(K)}{\mathcal{C}(K)},
\end{align*}
where $\mu$ is an arbitrary measure on $\mathcal{R}$ with total variation $|\mu|$. The space $\mathfrak{M}$ consists of all measures $\mu$ on $\mathcal{R}$ with finite $\|\mu\|_{\mathfrak{M}}$. Let $0<p<\infty$ and $0<q<\infty$. Denote by $L^{p,q}(\mathcal{C})$ the capacitary Lorentz space which consists of functions $f$ such that
\begin{align*}
\|f\|_{L^{p,q}(\mathcal{C})}=p^{\frac{1}{q}}\left(\int_{0}^{\infty}\mathcal{C}(\{x\in\mathcal{R}:|f(x)|>t\})^{\frac{q}{p}}t^{q}\frac{dt}{t}\right)^{\frac{1}{q}}
\end{align*}
are finite (see also \cite{OK3}). Let $\mathcal{L}^{p,q}(\mathcal{C})$ be the completion of $C_{0}$ with respect to the quasi-norm $\left\|\cdot\right\|_{L^{p,q}(\mathcal{C})}$, i.e., 
\begin{align*}
\mathcal{L}^{p,q}(\mathcal{C})=\overline{C_{0}}^{L^{p,q}(\mathcal{C})},
\end{align*}
where the closure is taken in $\left\|\cdot\right\|_{L^{p,q}(\mathcal{C})}$. Then the following theorem extends the result of \cite[Theorem 1.3]{OK}.
\begin{theorem}\label{pre theorem}
Let $0<q\leq 1$. Then the following isomorphism
\begin{align*}
\mathcal{L}^{1,q}(\mathcal{C})^{\ast}\approx\mathfrak{M}
\end{align*}
holds. Moreover, it holds that 
\begin{align}\label{1.2}
\|\mu\|_{\mathfrak{M}}=\inf\{a>0: |\mu|^{\ast}(E)\leq a\cdot\mathcal{C}(E)\},
\end{align}
where the infimum is taken over all subsets $E$ of $\mathcal{R}$.
\end{theorem}
Note that the dual space of $\mathcal{L}^{p,p}(\mathcal{C})=\mathcal{L}^{p}(\mathcal{C})$ is studied in \cite{OK} for $1<p<\infty$. The above theorem can be seen as an analogue of the classic fact that $(L^{1,q})^{\ast}\approx L^{\infty}$. Furthermore, the norm $\left\|\cdot\right\|_{\mathfrak{M}}$ is a variant of trace class inequality studied in \cite{MV}. For other aspects of Lorentz type capacities, we refer the reader to \cite{CM}, \cite{CMS}, \cite{CMS12}, and \cite{CS}.

The paper is organized as follows: Section $2$ addresses the dense subspace and normability of Lorentz spaces $L^{p,q}(\mu)$, while Section $3$ provides a brief review of Banach function space theory. The proofs of the main results are given in Section $4$. In what follows, $C(\alpha,\beta,\gamma,\ldots)$ denotes a positive constant depending only on the indicated parameters with the follow-up constants denoted by $C'(\alpha,\beta,\gamma,\ldots)$, $C''(\cdots)$, etc. We also abbreviate $\{x\in\mathcal{R}:P(x)\}$ as $\{P\}$, where $P(\cdot)$ is a statement. For instance, $\{|f|>t\}$ refers to $\{x\in\mathcal{R}:|f(x)|>t\}$. Suppose that $X$ is a locally compact Hausdorff space. We say that $\mu$ is a measure on $X$ if $\mu$ is a continuous linear functional on the space $C_{0}(X)$ of all compactly supported continuous functions on $X$. Here $C_{0}(X)$ is equipped with the inductive limit of the topologies of uniform convergence on the subspaces $C_{0}(X,K)$, where $K\subseteq X$ is compact and $C_{0}(X,K)$ consists of all continuous functions supported in $K$. The related measure theory is in accord with the N. Bourbaki approach, see \cite{BN} for further details. A more concise reference to this measure theory is given by \cite{DN}.

\section{Preliminaries on Lorentz Spaces}
We will use the following isomorphisms in the later sections, where the proofs can be found in \cite[Theorem 1.4.16]{GL}.
\begin{enumerate}
\item $L^{p,q}(\mu)^{\ast}\approx L^{\infty}(\mu)$ when $p=1$, $0<q\leq 1$.

\medskip
\item $L^{p,q}(\mu)^{\ast}\approx L^{p',\infty}(\mu)$ when $1<p<\infty$, $0<q\leq 1$.

\medskip
\item $L^{p,q}(\mu)^{\ast}\approx L^{p',q'}(\mu)$ when $1<p<\infty$, $1<q<\infty$.
\end{enumerate}
Note that the isomorphisms $\approx$ in (1) and (3) become isometric for $p=q$. The next result, the density of $C_{0}$ in $L^{p,q}(\mu)$, is surely well known to specialists, but we have not been able to find a reference for its proof. Since $L^{p,q}(\mu)$ is not normable for certain ranges of $p,q$, the density cannot be derived by \cite[Theorem 6.3.18]{PKJF}, which asserts that $C_{0}$ is dense in $X$, where $X$ is any Banach function space with absolutely continuous norm.
\begin{proposition}\label{lorentz dense}
Let $X$ be a locally compact Hausdorff space and $0<p,q<\infty$. Assume that $\mu$ is a $\sigma$-finite positive measure on $X$ that
\begin{align*}
X=\bigcup_{n=1}^{\infty}K_{n},
\end{align*}
where $K_{n}$ is compact with $\mu(K_{n})<\infty$, $n\in\mathbb{N}$. Denote by $L^{p,q}(\mu)$ the corresponding Lorentz space and $C_{0}$ the set of all compactly supported continuous functions on $X$. Then $C_{0}$ is dense in $L^{p,q}(\mu)$.
\end{proposition}

\begin{proof}
We first prove the set $\mathcal{K}$ of all compactly supported functions is dense in $L^{p,q}(\mu)$. Let $f\in L^{p,q}(\mu)$ and 
\begin{align*}
E_{n}=\bigcup_{i=1}^{n}K_{i},\quad n\in\mathbb{N}.
\end{align*}
For each $n\in\mathbb{N}$, the function $f\chi_{E_{n}}$ is compactly supported. Since $f\in L^{p,q}(\mu)$, we may assume that $f$ is real-valued. Besides, it is easy to argue that $\mu(\{|f|>t\})<\infty$ for each $0<t<\infty$. As a consequence,  $f-f\chi_{E_{n}}$ converges to $0$ pointwise as $n\rightarrow\infty$. The observation that 
\begin{align*}
\mu(\{|f-f\chi_{E_{n}}|>t\})\leq\mu(\{|f|>t\}),\quad 0<t<\infty,\quad n\in\mathbb{N}
\end{align*}
entails
\begin{align*}
\|f-f\chi_{E_{n}}\|_{L^{p,q}(\mu)}^{q}=p\int_{0}^{\infty}\mu(\{|f-f\chi_{E_{n}}|>t\})^{\frac{q}{p}}t^{q}\frac{dt}{t}\rightarrow 0,\quad n\rightarrow\infty
\end{align*}
by Lebesgue dominated convergence theorem. The density of $\mathcal{K}$ in $L^{p,q}(\mu)$ follows.

Subsequently, we prove that $C_{b}$ is dense in $\mathcal{K}$, where $C_{b}$ is the set of all bounded continuous functions on $X$. Let $f\in\mathcal{K}$ and $f_{M}=\max(\min(f,M),-M)$, $M>0$. By change of variables, we have
\begin{align*}
\|f_{M}-f\|_{L^{p,q}(\mu)}^{q}&=p\int_{0}^{\infty}\mu(\{|f_{M}-f|>t\})^{\frac{q}{p}}t^{q}\frac{dt}{t}\\
&=\frac{p}{q}\int_{0}^{\infty}\mu(\{|f_{M}-f|>t^{\frac{1}{q}}\})^{\frac{q}{p}}dt.
\end{align*}
Note that 
\begin{align*}
&\{|f_{M}-f|>t^{\frac{1}{q}}\}\\
&=\{\{f>M\}\cap\{|f_{M}-f|>t^{\frac{1}{q}}\}\}\cup\{\{f<-M\}\cap\{|f_{M}-f|>t^{\frac{1}{q}}\}\}\\
&=\{\{f>M\}\cap\{f-M>t^{\frac{1}{q}}\}\}\cup\{\{f<-M\}\cap\{-M-f>t^{\frac{1}{q}}\}\}\\
&=\{\{f>M\}\cap\{|f|>M+t^{\frac{1}{q}}\}\}\cup\{\{f<-M\}\cap\{|f|>M+t^{\frac{1}{q}}\}\}\\
&=\{\{|f|>M\}\cap\{|f|>M+t^{\frac{1}{q}}\}\}\\
&=\{|f|>M+t^{\frac{1}{q}}\}.
\end{align*}
Using the elementary inequality that
\begin{align*}
(a+b)^{r}\leq C(r)(a^{r}+b^{r}),\quad C(r)=\max\{2^{r-1},1\}
\end{align*}
with $r=1/q$, $a^{r}=M$, and $b=t$, one has 
\begin{align*}
\{|f|>M+t^{\frac{1}{q}}\}\subseteq\{|f|>C'(q)(M^{q}+t)^{\frac{1}{q}}\},
\end{align*}
and hence
\begin{align*}
\|f_{M}-f\|_{L^{p,q}(\mu)}^{q}&=\frac{p}{q}\int_{0}^{\infty}\mu(\{|f|>M+t^{\frac{1}{q}}\})^{\frac{q}{p}}dt\\
&\leq\frac{p}{q}\int_{0}^{\infty}\mu(\{|f|>C'(q)(M^{q}+t)^{\frac{1}{q}}\})^{\frac{q}{p}}dt\\
&=\frac{p}{q}\int_{M^{q}}^{\infty}\mu(\{|f|>C'(q)t^{\frac{1}{q}}\})^{\frac{q}{p}}dt\\
&=\frac{p}{qC''(q)}\int_{C''(q)M^{q}}^{\infty}\mu(\{|f|>t^{\frac{1}{q}}\})^{\frac{q}{p}}dt.
\end{align*}
Since
\begin{align*}
\|f\|_{L^{p,q}(\mu)}^{q}=p\int_{0}^{\infty}\mu(\{|f|>t\})^{\frac{q}{p}}t^{q}\frac{dt}{t}=\frac{p}{q}\int_{0}^{\infty}\mu(\{|f|>t^{\frac{1}{q}}\})^{\frac{q}{p}}dt<\infty,
\end{align*}
we obtain $\|f_{M}-f\|_{L^{p,q}(\mu)}^{q}\rightarrow 0$ as $M\rightarrow\infty$. For any $\varepsilon>0$, choose an $M>0$ such that $\|f_{M}-f\|_{L^{p,q}(\mu)}<\varepsilon$. Since $X$ is $\sigma$-finite, we can assume that the sequence $\{K_{n}\}_{n=1}^{\infty}$ is increasing. Pick an $n_{0}\in\mathbb{N}$ such that the interior $K_{n_{0}}^{\circ}$ of $K_{n_{0}}$ contains the support of $f$. Since $f_{M}$ is measurable and $K_{n_{0}}^{\circ}$ is of finite measure, by Lusin's theorem, there is a closed set $F\subseteq K_{n_{0}}^{\circ}$ such that $\mu(K_{n_{0}}^{\circ}\setminus F)<\varepsilon^{p}$ and $f_{M}|_{F}$ is a continuous function. Note that $f_{M}$ is identically zero on the closed set $\widetilde{F}=X\setminus K_{n_{0}}^{\circ}$. By gluing lemma, the function $f_{M}|_{F\cup\widetilde{F}}$ is continuous. Since $X$ is locally compact Hausdorff, we can choose by Tietze extension theorem a continuous function $v$ such that $|v|\leq M$ and $v=f_{M}$ on $F\cup\widetilde{F}$. Then
\begin{align*}
\|v-f_{M}\|_{L^{p,q}(\mu)}^{q}&=p\int_{0}^{\infty}\mu(\{|v-f_{M}|>t\})^{\frac{q}{p}}t^{q}\frac{dt}{t}\\
&=p\int_{0}^{\infty}\mu(\{x\in K_{n_{0}}^{\circ}\setminus F:|v(x)-f_{M}(x)|>t\})^{\frac{q}{p}}t^{q}\frac{dt}{t}\\
&=p\int_{0}^{2M}\mu(\{x\in K_{n_{0}}^{\circ}\setminus F:|v(x)-f_{M}(x)|>t\})^{\frac{q}{p}}t^{q}\frac{dt}{t}\\
&\leq C(p,q,M)\mu(K_{n_{0}}^{\circ}\setminus F)^{\frac{q}{p}}\\
&<C(p,q,M)\varepsilon^{q}.
\end{align*}
As a result, we have 
\begin{align*}
\|f-v\|_{L^{p,q}(\mu)}\leq\kappa_{0}(\|f-f_{M}\|_{L^{p,q}(\mu)}+\|v-f_{M}\|_{L^{p,q}(\mu)})<C(p,q,M,\kappa_{0})\varepsilon,
\end{align*}
where $\kappa_{0}=C(p,q)$ and the density of $C_{b}$ in $L^{p,q}(\mu)$ now follows.

Now we claim that $C_{0}$ is dense in $C_{b}$. To this end, let $v$ be continuous and $|v|\leq M$ for some $M>0$. For each $N\in\mathbb{N}$, let $O_{N}=\{|v|>1/N\}$. Then $O_{N}$ is open and 
\begin{align*}
\mu(O_{N})\leq C(p,q,N)\|f\|_{L^{p,q}(\mu)}<\infty.
\end{align*}
We observe that
\begin{align*}
\|v\chi_{O_{N}^{c}}\|_{L^{p,q}(\mu)^{q}}&=p\int_{0}^{\infty}\mu(\{x\in O_{N}^{c}:|v(x)|>t\})^{\frac{q}{p}}t^{q}\frac{dt}{t}\\
&=p\int_{0}^{\frac{1}{N}}\mu(\{x\in O_{N}^{c}:|v(x)|>t\})^{\frac{q}{p}}t^{q}\frac{dt}{t}\\
&\leq p\int_{0}^{\frac{1}{N}}\mu(\{|v|>t\})^{\frac{q}{p}}t^{q}\frac{dt}{t}\\
&\rightarrow 0
\end{align*}
as $N\rightarrow\infty$. Thus, for any $\varepsilon>0$, there is an open set $O\subseteq\mathcal{R}$ such that $\mu(O)<\infty$ and $\|v\chi_{O^{c}}\|_{L^{p,q}(\mu)}<\varepsilon$. Since $\mu$ is additive, we have 
\begin{align*}
\sum_{n=1}^{\infty}\mu(O\cap E_{n})=\mu(O)<\infty,
\end{align*}
where the sequence $\{E_{n}\}_{n=1}^{\infty}\subseteq X$ is chosen such that $\{E_{n}\}_{n=1}^{\infty}$ is disjoint and $O\cap E_{n}$ is precompact for each $n\in\mathbb{N}$. There is some $n_{0}\in\mathbb{N}$ such that
\begin{align*}
G=O\cap\left(\bigcup_{n=n_{0}}^{\infty}E_{n}\right)
\end{align*}
satisfies
\begin{align*}
\mu(G)\leq\sum_{n=n_{0}}^{\infty}\mu(O\cap E_{n})<\varepsilon^{q}.
\end{align*}
Denote by 
\begin{align*}
H=O\cap\left(\bigcup_{n=1}^{n_{0}}E_{n}\right).
\end{align*}
Since $X$ is locally compact Hausdorff and $H$ is precompact, we can choose by Urysohn's lemma an $\eta\in C_{0}$ such that $0\leq\eta\leq 1$ and $\eta=1$ on $H$. We have
\begin{align*}
&\|\eta v-v\|_{L^{p,q}(\mu)}\\
&\leq(\kappa_{0}^{2}+1)(\|(\eta v-v)\chi_{O^{c}}\|_{L^{p,q}(\mu)}+\|(\eta v-v)\chi_{G}\|_{L^{p,q}(\mu)}+\|(\eta v-v)\chi_{H}\|_{L^{p,q}(\mu)})\\
&=(\kappa_{0}^{2}+1)(\|(\eta v-v)\chi_{O^{c}}\|_{L^{p,q}(\mu)}+\|(\eta v-v)\chi_{G}\|_{L^{p,q}(\mu)}\|_{L^{p,q}(\mu)})
\end{align*}
as $\eta=1$ on $H$. On the other hand, one has
\begin{align*}
\|(\eta v-v)\chi_{O^{c}}\|_{L^{p,q}(\mu)}^{q}&=p\int_{0}^{\infty}\mu(\{x\in O^{c}:|\eta(x)v(x)-v(x)|>t\})^{\frac{q}{p}}t^{q}\frac{dt}{t}\\
&\leq p\int_{0}^{\infty}\mu\left(\left\{x\in O^{c}:|v(x)|>\frac{t}{2}\right\}\right)^{\frac{q}{p}}t^{q}\frac{dt}{t}\\
&\leq C(p,q)\|v\chi_{O^{c}}\|_{L^{p,q}(\mu)}\\
&<C(p,q)\varepsilon^{q}.
\end{align*}
Moreover, we have
\begin{align*}
\|(\eta v-v)\chi_{G}\|_{L^{p,q}(\mu)}^{q}&\leq p\int_{0}^{2M}\mu(\{x\in G:|\eta(x)v(x)-v(x)|>t\})^{\frac{q}{p}}t^{q}\frac{dt}{t}\\
&\leq C(p,q,M)\mu(G)\\
&<C(p,q,M)\varepsilon^{q},
\end{align*}
we conclude that $\|\eta v-v\|_{L^{p,q}(\mu)}<C'(p,q,M,\kappa_{0})\varepsilon$, which finishes the proof as $\eta v\in C_{0}$.
\end{proof}

Now we address the normability of $L^{p,q}(\mu)$. Let $1<p<\infty$, $1\leq q<\infty$, $0<r\leq 1$, and $\mu\geq 0$ be  $\sigma$-finite with $\mu(\mathcal{R})=\infty$. For any $f\in\mathcal{M}(\mathcal{R},\mu)$, define
\begin{align*}
&f_{r}^{\ast\ast}(t)=\sup_{0<\mu(E)<\infty}\left(\frac{1}{\mu(E)}\int_{E}|f|^{r}d\mu\right)^{\frac{1}{r}},\quad 0<t<\infty,\\
&\Gamma_{r}^{p,q}(f)=\left(\int_{0}^{\infty}(t^{\frac{1}{p}}f_{r}^{\ast\ast}(t))^{q}\frac{dt}{t}\right)^{\frac{1}{q}},
\end{align*}
where the supremum in $f_{r}^{\ast\ast}$ is taken over all $\mu$-measurable subsets $E$ of $\mathcal{R}$ with $0<\mu(E)<\infty$. Then it is noted in \cite[Exercise 1.4.3]{GL} that 
\begin{align*}
\|f\|_{L^{p,q}(\mu)}\leq\Gamma_{r}^{p,q}(f)\leq\left(\frac{p}{p-r}\right)^{\frac{1}{r}}\|f\|_{L^{p,q}(\mu)},
\end{align*}
and $\Gamma_{r}^{p,q}(\cdot)$ is a norm when $r=1$. In other words, the Lorentz spaces $L^{p,q}(\mu)$ are normable when $1<p<\infty$ and $1\leq q<\infty$, we will use this fact from frequently in the later sections.

We conclude this section by presenting the following simple yet useful formula.
\begin{align*}
\||f|^{r}\|_{L^{p,q}(\mu)}=\|f\|_{L^{pr,qr}(\mu)}^{r},
\end{align*}
where $0<p<\infty$, $0<q\leq\infty$, and $0<r<\infty$.

\section{Preliminaries on Banach Function Spaces}
Assume that $\mu\geq 0$ is a $\sigma$-finite measure and $\rho:\mathcal{M}(\mathcal{R},\mu)\rightarrow[0,\infty]$ is a functional. We say that $\rho$ is a Banach function norm provided that for every $f,g\in\mathcal{M}(\mathcal{R},\mu)$, $\lambda\geq 0$, and $\mu$-measurable sets $E\subseteq\mathcal{R}$, the following conditions are satisfied. 
\begin{enumerate}
\item $\rho(f)=\rho(|f|)$ and $\rho(f)=0$ if and only if $f=0$ $\mu$-almost-everywhere.

\medskip
\item $\rho(\lambda f)=\lambda\rho(f)$.

\medskip
\item $\rho(f+g)\leq\rho(f)+\rho(g)$.

\medskip
\item If $0\leq g\leq f$ $\mu$-almost-everywhere, then $\rho(g)\leq\rho(f)$.

\medskip
\item If $0\leq f_{1}\leq f_{2}\leq\cdots$, $f_{n}\uparrow f$ $\mu$-almost-everywhere, then $\rho(f_{n})\uparrow\rho(f)$. This is also known to be the Fatou's property.

\medskip
\item If $E$ is bounded, then $\rho(\chi_{E})<\infty$.

\medskip
\item If $E$ is bounded, then there exists a positive constant $C(E)$ such that 
\begin{align*}
\int_{E}|f|d\mu\leq C(E)\rho(f)
\end{align*}
holds for all $f\in\mathcal{M}(\mathcal{R},\mu)$.
\end{enumerate}
Then we denote by $X=X(\rho)$ the set of all functions $f\in\mathcal{M}(\mathcal{R},\mu)$ such that $\rho(f)<\infty$ and call $X$ the Banach function space. For any $f\in\mathcal{M}(\mathcal{R},\mu)$, let
\begin{align*}
\|f\|_{X}=\rho(f).
\end{align*}
Then $\left\|\cdot\right\|_{X}$ is a norm on $X$. It can be shown that $X$ is a Banach space (see \cite[Corollary 6.1.15]{PKJF}). Moreover, the second K\"othe dual $X''=(X')'$ of $X$ satisfies that
\begin{align*}
X''=X
\end{align*}
(see \cite[Theorem 6.2.9]{PKJF}). The crucial properties used in the proof of the above isometric isomorphism are the triangle inequality $\|f+g\|_{X}\leq\|f\|_{X}+\|g\|_{X}$ and the Fatou's property.

Moreover, given a Banach function space $X$, we say that $X$ has an absolutely continuous norm provided that for every sequence $\{f_{n}\}_{n=1}^{\infty}\subseteq\mathcal{M}(\mathcal{R},\mu)$, if $f_{n}\downarrow 0$ $\mu$-almost-everywhere, then $\|f_{n}\|_{X}\downarrow 0$. The following proposition is tacitly used in the proof of \cite[Proposition 2.8]{OP}, we include its proof for the sake of completeness.
\begin{proposition}\label{use convex}
Let $1\leq q<\infty$ and $X$ be a Banach function space. If $X$ is $q$-concave in the sense that
\begin{align*}
\left(\sum_{i=1}^{m}\|f_{i}\|_{X}^{q}\right)^{\frac{1}{q}}\leq\mathfrak{C}\left\|\left(\sum_{i=1}^{m}|f_{i}|^{q}\right)^{\frac{1}{q}}\right\|_{X}
\end{align*}
holds for some positive constant $\mathfrak{C}$ and $\{f_{i}\}_{i=1}^{m}$ is any finite sequence of functions in $X$, then $X$ has an absolutely continuous norm. In which case, it holds that
\begin{align*}
X'=X^{\ast}.
\end{align*}
\end{proposition}

\begin{proof}
Let $\{f_{n}\}_{n=1}^{\infty}$ be a sequence in $X$ with $f_{n}\downarrow 0$ $\mu$-almost-everywhere. We first claim that if $f_{n}\rightarrow\varphi$ in $X$, then $\varphi=0$ $\mu$-almost everywhere. To see this, we use the Fatou's property of $X$ to obtain
\begin{align*}
\|\varphi\|_{X}=\left\|\lim_{n\rightarrow\infty}f_{n}(\cdot)-\varphi\right\|_{X}\leq\liminf_{n\rightarrow\infty}\|f_{n}-\varphi\|_{X}=0,
\end{align*}
which yields the claim.

Suppose to the contrary that $\|f_{n}\|_{X}\rightarrow 0$ fails. As $X$ is a Banach space, the previous argument shows that $\{f_{n}\}_{n=1}^{\infty}$ is not Cauchy. Then there are $\varepsilon>0$ and $n_{1}<n_{2}<\cdots$ such that 
\begin{align*}
\|f_{n_{k+1}}-f_{n_{k}}\|_{X}\geq\varepsilon,\quad k=1,2,\ldots.
\end{align*}
Using the $q$-concavity of $X$, one has for fixed $m\in\mathbb{N}$ that
\begin{align*}
\left(\sum_{k=1}^{m}\|f_{n_{k+1}}-f_{n_{k}}\|_{X}^{q}\right)^{\frac{1}{q}}&\leq\mathfrak{C}\left\|\left(\sum_{k=1}^{m}|f_{n_{k+1}}-f_{n_{k}}|^{q}\right)^{\frac{1}{q}}\right\|_{X}\\
&\leq\mathfrak{C}\left\|\sum_{k=1}^{m}|f_{n_{k+1}}-f_{n_{k}}|\right\|_{X}\\
&=\mathfrak{C}\|f_{n_{1}}-f_{n_{m+1}}\|_{X},
\end{align*}
and hence
\begin{align*}
\mathfrak{C}\|f_{n_{1}}-f_{n_{m+1}}\|_{X}\geq m^{\frac{1}{q}}\varepsilon,\quad m=1,2,\ldots.
\end{align*}
Note that $f_{n_{1}}-f_{n_{m}}\uparrow f_{n_{1}}$ $\mu$-almost-everywhere. Taking limit $m\rightarrow\infty$ to the above inequality and using Fatou's property of $X$ again, we obtain a contradiction.
\end{proof}

We end this section by noting that if $\left\|\cdot\right\|_{X}$ is only a quasi-norm with weak Fatou's property, the second K\"othe dual $X''$ of $X$ does not necessarily isomorphic to $X$. This is the main difficulty in proving the isomorphisms in Theorem \ref{dual N}. It is worth noting that $\left\|\cdot\right\|_{L^{p,q}(\mu)}$ is a quasi-norm which fails to satisfy the triangle inequality. Hence if we equip $L^{p,q}(\mu)$ with $\left\|\cdot\right\|_{L^{p,q}(\mu)}$, then $L^{p,q}(\mu)$ is surely not a Banach function space under this quasi-norm. Nevertheless, the Fatou's property clearly holds for $\left\|\cdot\right\|_{L^{p,q}(\mu)}$ and we are able to recover $L^{p,q}(\mu)$ to be a Banach function space by switching to another norm. Let $1<p<\infty$ and $1\leq q<\infty$. Consider the norm $\Gamma_{1}^{p,q}(\cdot)$ that introduced in the previous section. Note that $\Gamma_{1}^{p,q}(\cdot)$ satisfies the Fatou's property. It is now clear that $L^{p,q}(\mu)$, equipped with $\Gamma_{1}^{p,q}(\cdot)$, is a Banach function space. For such $p,q$, it is easy to see that 
\begin{align*}
\|f_{1}+f_{2}+\cdots\|_{L^{p,q}(\mu)}\leq C(p,q)(\|f_{1}\|_{L^{p,q}(\mu)}+\|f_{2}\|_{L^{p,q}(\mu)}+\cdots),
\end{align*}
where $\{f_{n}\}_{n=1}^{\infty}\subseteq\mathcal{M}(\mathcal{R},\mu)$ and $f_{1}+f_{2}+\cdots$ is the usual pointwise sum.

\section{Proofs of Main Results}
We first establish that for the Bessel capacities $\mathcal{C}={\rm Cap}_{\alpha,s}(\cdot)$ satisfy the following estimates for arbitrary measurable sets $E\subseteq\mathcal{R}$.
\begin{align}\label{sobolev 1}
|E|^{\varepsilon}\leq C(n,\alpha,s,\varepsilon)\mathcal{C}(E),\quad 0<\varepsilon\leq 1,\quad \alpha s=n
\end{align}
and
\begin{align}\label{sobolev 2}
|E|^{\varepsilon}\leq C(n,\alpha,s,\varepsilon)\mathcal{C}(E),\quad\frac{n-\alpha s}{n}\leq\varepsilon\leq 1,\quad \alpha s<n.
\end{align}
Assume that $\alpha s=n$. Using Young's convolution inequality, we have 
\begin{align*}
\|G_{\alpha}\ast f\|_{L^{r}}\leq\|G_{\alpha}\|_{L^{\overline{r}}}\|f\|_{L^{s}},
\end{align*}
where
\begin{align*}
\frac{1}{r}+1=\frac{1}{\overline{r}}+\frac{1}{s},\quad 1<s=\frac{n}{\alpha}\leq r<\infty.
\end{align*}
Then $0<\overline{r}=nr/(n+r(n-\alpha))\leq 1$ and hence $\|G_{\alpha}\|_{L^{\overline{r}}}<\infty$. Suppose that $G_{\alpha}\ast f\geq 1$ on $E$. Then 
\begin{align*}
|E|^{s/r}\leq C(n,\alpha,s,r)\|f\|_{L^{s}}^{s},
\end{align*}
which gives (\ref{sobolev 1}) by simple algebraic manipulations.

Now assume that $\alpha s<n$. By Sobolev embedding theorem, we have 
\begin{align*}
\|G_{\alpha}\ast f\|_{L^{r}}\leq C(n,\alpha,s,r)\|f\|_{L^{s}},
\end{align*}
where
\begin{align*}
s\leq r\leq s^{\ast},\quad\frac{1}{s^{\ast}}=\frac{1}{s}-\frac{\alpha}{n}=\frac{n-\alpha s}{ns}.
\end{align*}
Suppose that $G_{\alpha}\ast f\geq 1$ on $E$. Then
\begin{align*}
|E|^{\frac{s}{r}}\leq C(n,\alpha,s,r)\|f\|_{L^{s}}^{s},\quad\frac{n-\alpha s}{n}\leq\frac{s}{r}\leq 1,
\end{align*}
which yields (\ref{sobolev 2}). A direct consequence of (\ref{sobolev 1}) and (\ref{sobolev 2}) is the absolute continuity for $\alpha s\leq n$, i.e., if $\mathcal{C}(E)=0$, then $|E|=0$. In view of this, the capacities $\mathcal{C}$ are finer than the Lebesgue measure.

We will use some classical results from nonlinear potential theory. Let $E\subseteq\mathcal{R}$ be an arbitrary set with $0<\mathcal{C}(E)<\infty$. Then there exists a positive measure $\mu=\mu^{E}$ on $\mathcal{R}$ such that $\mu$ is supported in $\overline{E}$ with 
\begin{align*}
&V^{E}(x)\geq 1\quad\text{for all }x\in E\setminus N,\\
&V^{E}(x)\leq 1\quad\text{for all }x\in{\rm supp}(\mu),\\
&\mu(\mathcal{R})=\int_{\mathcal{R}}(G_{\alpha}\ast\mu)(x)^{s'}dx=\int_{\mathcal{R}}V^{E}(x)d\mu(x)=\mathcal{C}(E),
\end{align*}
where $N\subseteq\mathcal{R}$ is some set with $\mathcal{C}(N)=0$ and $V^{E}=G_{\alpha}\ast(G_{\alpha}\ast\mu)^{s'-1}$ (see \cite[Theorem 2.5.6]{AH}). We call $V^{E}$ the nonlinear potential associated with $E$. Further, it is proved in \cite[Lemma 3.2]{OP} that there exists a positive constant $\delta=C(n,\alpha,s)$ such that $(V^{E})^{\delta}\in L^{1}(\mathcal{C})$ with 
\begin{align*}
C(n,\alpha,s,\delta)^{-1}\mathcal{C}(E)\leq\|(V^{E})^{\delta}\|_{L^{1}(\mathcal{C})}\leq C(n,\alpha,s,\delta)\mathcal{C}(E).
\end{align*}
Moreover, \cite[Theorem 3.1]{OP} gives $(V^{E})^{\delta}\in A_{1}^{\rm loc}$ with 
\begin{align}\label{universal}
[(V^{E})^{\delta}]_{A_{1}^{\rm loc}}\leq\mathfrak{c},
\end{align}
where $\mathfrak{c}=C(n,\alpha,s)$ is a positive constant. We will reserve this symbol $\mathfrak{c}$ throughout the paper. Finally, we note that $L^{1}(\mathcal{C})$ is normable (see \cite[Theorem 2.8]{OK2}).

\begin{proof}[Proof of Theorem \ref{first predual}]
Let $1<r<p$ and $1<r\leq q$. We first prove that $M^{p,q}$ is $r$-convex, i.e.,
\begin{align*}
\left\|\left(\sum_{i=1}^{m}|f_{i}|^{r}\right)^{\frac{1}{r}}\right\|_{M^{p,q}}\leq\kappa\left(\sum_{i=1}^{m}\|f_{i}\|_{M^{p,q}}^{r}\right)^{\frac{1}{r}}
\end{align*}
holds for every finite sequence $\{f_{i}\}_{i=1}^{m}$ of functions in $M^{p,q}$, where $\kappa=C(p,q,r)$. For such a sequence $\{f_{i}\}_{i=1}^{m}$, we have 
\begin{align*}
\left\|\left(\sum_{i=1}^{m}|f_{i}|^{r}\right)^{\frac{1}{r}}\right\|_{M^{p,q}}&=\sup_{K}\frac{1}{\mathcal{C}(K)^{1/q}}\left\|\left(\sum_{i=1}^{m}|f_{i}|^{r}\right)^{\frac{1}{r}}\chi_{K}\right\|_{L^{p,q}}\\
&=\sup_{K}\frac{1}{\mathcal{C}(K)^{1/q}}\left\|\sum_{i=1}^{m}|f_{i}|^{r}\chi_{K}\right\|_{L^{p/r,q/r}}^{\frac{1}{r}}\\
&\leq\kappa\cdot\sup_{K}\frac{1}{\mathcal{C}(K)^{1/q}}\left(\sum_{i=1}^{m}\left\||f_{i}|^{r}\chi_{K}\right\|_{L^{p/r,q/r}}\right)^{\frac{1}{r}}\\
&=\kappa\cdot\sup_{K}\frac{1}{\mathcal{C}(K)^{1/q}}\left(\sum_{i=1}^{m}\left\|f_{i}\chi_{K}\right\|_{L^{p,q}}^{r}\right)^{\frac{1}{r}}\\
&\leq\kappa\left(\sum_{i=1}^{m}\|f_{i}\|_{M^{p,q}}^{r}\right)^{\frac{1}{r}},
\end{align*}
which justifies the $r$-convexity of $M^{p,q}$.

Now we show that $(M^{p,q})'$ is $r'$-concave. Fix a finite sequence $\{g_{i}\}_{i=1}^{m}$ in $(M^{p,q})'$. Using $\ell^{r'}((M^{p,q})^{\ast})\approx[\ell^{r}(M^{p,q})]^{\ast}$ (see \cite[Theorem 18.42]{DN}), we infer that 
\begin{align*}
&\left(\sum_{i=1}^{m}\|g_{i}\|_{(M^{p,q})'}^{r'}\right)^{\frac{1}{r'}}\\
&=\left(\sum_{i=1}^{m}\|\mathcal{L}_{g_{i}}\|_{(M^{p,q})^{\ast}}^{r'}\right)^{\frac{1}{r'}}\\
&\leq\mathfrak{C}\cdot\sup_{\|\{f_{i}\}_{i=1}^{m}\|_{\ell^{r}(M^{p,q})}\leq 1}\left|\int_{\mathbb{R}^{n}}\sum_{i=1}^{m}f_{i}(x)g_{i}(x)dx\right|\\
&\leq\mathfrak{C}\cdot\sup_{\|\{f_{i}\}_{i=1}^{m}\|_{\ell^{r}(M^{p,q})}\leq 1}\int_{\mathbb{R}^{n}}\left(\sum_{i=1}^{m}|f_{i}(x)|^{r}\right)^{\frac{1}{r}}\left(\sum_{i=1}^{m}|g_{i}(x)|^{r'}\right)^{\frac{1}{r'}}dx\\
&\leq\mathfrak{C}\cdot\sup_{\|\{f_{i}\}_{i=1}^{m}\|_{\ell^{r}(M^{p,q})}\leq 1}\left\|\left(\sum_{i=1}^{m}|f_{i}|^{r}\right)^{\frac{1}{r}}\right\|_{M^{p,q}}\left\|\left(\sum_{i=1}^{m}|g_{i}|^{r'}\right)^{\frac{1}{r'}}\right\|_{(M^{p,q})'}\\
&\leq\mathfrak{C}\kappa\left\|\left(\sum_{i=1}^{m}|g_{i}|^{r'}\right)^{\frac{1}{r'}}\right\|_{(M^{p,q})'},
\end{align*}
where the positive constant $\mathfrak{C}$ is due to the isomorphism $\ell^{r'}((M^{p,q})^{\ast})\approx[\ell^{r}(M^{p,q})]^{\ast}$, and it depends on $n,\alpha,s,p,q$. Thus the $r'$-concavity of $(M^{p,q})'$ follows.

The constants $\mathfrak{C},k$ can be reduced to $1$ if $p=q$ or $M^{p,q}$ is replaced by the normed space $\mathfrak{M}^{p,q}$, where 
\begin{align*}
\|f\|_{\mathfrak{M}^{p,q}}=\sup_{K}\frac{\Gamma_{1}^{p,q}(f\chi_{K})}{\mathcal{C}(K)^{1/q}},\quad f\in\mathcal{M}(\mathcal{R}).
\end{align*}
Since $\mathfrak{M}^{p,q}$ is a Banach function space, Proposition \ref{use convex} entails
\begin{align*}
[(\mathfrak{M}^{p,q})']^{\ast}=(\mathfrak{M}^{p,q})''=\mathfrak{M}^{p,q}.
\end{align*}
Switching $\mathfrak{M}^{p,q}$ back to $M^{p,q}$, we obtain the result of this theorem for $M^{p,q}$. The proof of $\mathcal{M}^{p,q}$ is argued exactly in the same fashion. 
\end{proof}

Now we show for the embedding that
\begin{align*}
M^{p,r}\hookrightarrow M^{p,q}
\end{align*}
for admissible exponents $p,q,r$. To this end, we need the following technical lemma, which is due to Strichartz \cite{ST} (see also \cite[Theorem 3.1.2]{MS}).
\begin{lemma}\label{IVstr}
Let $\{\mathcal{B}^{(j)}\}_{j\geq 0}$ be a covering of $\mathbb{R}^{n}$ by balls with unit diameter. Assume that this covering has finite multiplicity that depends only on $n$. Further, let $O^{(j)}$ be the center of $\mathcal{B}^{(j)}$, $O^{(0)}=0$, and $\eta_{j}=\eta(x-O^{(j)})$, where $\eta$ is infinitely differentiable function that supported in $B_{1}(0)$ and $\eta=1$ on $\mathcal{B}^{(0)}$. Then
\begin{align*}
\|u\|_{W^{\alpha,s}}\approx\left(\sum_{j\geq 0}\|u\eta_{j}\|_{W^{\alpha,s}}^{s}\right)^{1/s}.
\end{align*}
Here $u\in W^{\alpha,s}$ if and only if $u=G_{\alpha}\ast f$ for some $f\in L^{s}$ and $\|u\|_{W^{\alpha,s}}=\|f\|_{L^{s}}$.
\end{lemma}

Note that the function $f\in L^{s}$ in $u=G_{\alpha}\ast f\in W^{\alpha,s}$ is uniquely determined; see \cite[Section 3.3]{SE} for a proof. The Strichartz lemma entails the localization of capacity, which reads as follows; see also \cite[Theorem 4]{AD} for a variant.
\begin{proposition}\label{IVSTR}
For any set $E\subseteq\mathcal{R}$, it holds that 
\begin{align}\label{IVSTRuse}
\mathcal{C}(E)\approx\sum_{j\geq 0}\mathcal{C}(E\cap\mathcal{B}^{(j)}),
\end{align}
where $\{\mathcal{B}_{j}\}_{j\geq 0}$ is as in Lemma \ref{IVstr}.
\end{proposition}

\begin{proof}
First we note that by the finite multiplicity of $\{\mathcal{B}^{(j)}\}_{j\geq 0}$ and the countable subadditivity of $\mathcal{C}$, one has
\begin{align*}
\mathcal{C}(E)\leq C(n,\alpha,s)\sum_{j\geq 0}\mathcal{C}(E\cap\mathcal{B}^{(j)}).
\end{align*}
To complete the proof of this proposition, we need an auxiliary capacity which defined as follows. For any compact set $K\subseteq\mathcal{R}$, define
\begin{align*}
\mathcal{C}'(K)=\inf\{\|\varphi\|_{W^{\alpha,s}}^{s}: \varphi\in\mathcal{S}(\mathcal{R}),~\varphi\geq\chi_{K}\},
\end{align*}
where $W^{\alpha,s}$ is introduced in Lemma \ref{IVstr} and $\mathcal{S}(\mathcal{R})$ is the Schwartz class; see \cite[Definition 2.2.1]{GL} for the definition of $\mathcal{S}(\mathcal{R})$. For any open set $G\subseteq\mathcal{R}$, define
\begin{align*}
\mathcal{C}'(G)=\sup_{\substack{K~\text{compact}\\K\subseteq G}}\mathcal{C}'(K).
\end{align*}
If $E\subseteq\mathcal{R}$ is an arbitrary set, then we define
\begin{align*}
\mathcal{C}'(E)=\inf_{\substack{G~\text{open}\\E\subseteq G\subseteq\mathcal{R}}}\mathcal{C}'(G).
\end{align*}
It can be proved that $\mathcal{C}=\mathcal{C}'$; see \cite[Proposition 2.3.13]{AH} for details.

To prove this proposition, we first consider that $E=K$ is compact. Following the notations in Lemma \ref{IVstr}, for any $\varphi\in\mathcal{S}$ with $\varphi\geq\chi_{K}$, we have $\varphi\eta_{j}\geq\chi_{K\cap\mathcal{B}^{(j)}}$, and hence 
\begin{align*}
\left(\sum_{j\geq 0}\mathcal{C}(K\cap\mathcal{B}^{(j)})\right)^{1/s}\leq\left(\sum_{j\geq 0}\|\varphi\eta_{j}\|_{W^{\alpha,s}}^{s}\right)^{1/s}\leq C(n,\alpha,s)\|\varphi\|_{W^{\alpha,s}}.
\end{align*}
Then (\ref{IVSTRuse}) follows by taking infimum over all such $\varphi$.

Consider now that $E=G$ is open. There is an increasing sequence $\{K_{n}\}_{n=1}^{\infty}$ of compact sets in $\mathcal{R}$ such that 
\begin{align*}
G=\bigcup_{n=1}^{\infty}K_{n}.
\end{align*}
Then the Fatou's property of $\mathcal{C}$ gives
\begin{align*}
\sum_{j\geq 0}\mathcal{C}(G\cap\mathcal{B}^{(j)})&=\sum_{j\geq 0}\sup_{n\in\mathbb{N}}\mathcal{C}(K_{n}\cap\mathcal{B}^{(j)})\\
&=\sup_{n\in\mathbb{N}}\sum_{j\geq 0}\mathcal{C}(K_{n}\cap\mathcal{B}^{(j)})\\
&\leq C(n,\alpha,s)\sup_{n\in\mathbb{N}}\mathcal{C}(K_{n})\\
&=C(n,\alpha,s)\mathcal{C}(G),
\end{align*}
which yields (\ref{IVSTRuse}) for open sets $E=G$.

Finally, given arbitrary set $E\subseteq\mathcal{R}$, let $G\supseteq E$ be open. Then
\begin{align*}
\sum_{j\geq 0}\mathcal{C}(E\cap\mathcal{B}^{(j)})\leq\sum_{j\geq 0}\mathcal{C}(G\cap\mathcal{B}^{(j)})\leq C(n,\alpha,s)\mathcal{C}(G).
\end{align*} 
As a result, (\ref{IVSTRuse}) follows by the outer regularity of $\mathcal{C}$, which finishes the proof.
\end{proof}

\begin{proposition}\label{addition}
Let $1<p<\infty$ and $1<q<\infty$. Assume that either
\begin{align*}
1<r<q\leq p<\infty,\quad \alpha s=n
\end{align*}
or
\begin{align*}
1<r<q\leq p<\infty,\quad\frac{n-\alpha s}{n}\leq\frac{q}{p},\quad\alpha s<n.
\end{align*}
Then 
\begin{align*}
M^{p,r}\hookrightarrow M^{p,q}
\end{align*}
and
\begin{align*}
L^{\infty}\hookrightarrow M^{p,q}.
\end{align*}
In fact, it holds that 
\begin{align*}
C(n,\alpha,s,p,q)^{-1}\|f\|_{M^{p,q}}\leq\sup_{{\rm diam}(K)\leq 1}\frac{\|f\chi_{K}\|_{L^{p,q}}}{\mathcal{C}(K)^{1/q}}\leq C(n,\alpha,s,p,q)\|f\|_{M^{p,q}},
\end{align*}
where the supremum is taken over all compact sets $K\subseteq\mathcal{R}$ with nonzero capacities and diameters less than $1$.
\end{proposition}

\begin{proof}
Denote by
\begin{align*}
M=\sup_{{\rm diam}(K)\leq 1}\frac{\|f\chi_{K}\|_{L^{p,q}}}{\mathcal{C}(K)^{1/q}}
\end{align*}
for simplicity. Let $f\in\mathcal{M}(\mathcal{R})$ and $K\subseteq\mathcal{R}$ be compact. Suppose that $\{Q_{n}\}_{n=1}^{\infty}$ is a sequence of closed cubes with disjoint interior and unit diameters, whose union is the whole $\mathcal{R}$. Following the first part of the proof of Theorem \ref{first predual}, by denoting $f_{n}=f\chi_{Q_{n}}$, $n\in\mathbb{N}$, we obtain
\begin{align*}
\|f\chi_{K}\|_{L^{p,q}}&=\left\|\left(\sum_{n}|f_{n}|^{q}\right)^{\frac{1}{q}}\chi_{K}\right\|_{L^{p,q}}\\
&=\lim_{n\rightarrow\infty}\left\|\left(\sum_{i=1}^{n}|f_{i}|^{q}\right)^{\frac{1}{q}}\chi_{K}\right\|_{L^{p,q}}\\
&\leq C(p,q)\lim_{n\rightarrow\infty}\left(\sum_{i=1}^{n}\||f_{i}|^{q}\chi_{K}\|_{L^{p/q,1}}\right)^{\frac{1}{q}}\\
&=C(p,q)\left(\sum_{n}\|f_{n}\chi_{K}\|_{L^{p,q}}^{q}\right)^{\frac{1}{q}}\\
&\leq C(p,q)M\left(\sum_{n}\mathcal{C}(K\cap Q_{i})\right)^{\frac{1}{q}}.
\end{align*}
Using Proposition \ref{IVSTR}, one obtains
\begin{align}\label{strichartz}
\sum_{n}\mathcal{C}(K\cap Q_{i})\leq C(n,\alpha,s)\mathcal{C}(K).
\end{align}
Combining (\ref{strichartz}) with the above estimates, then $\|f\|_{M^{p,q}}\leq C(n,\alpha,s,p,q)M$. Note that $M\leq\|f\|_{M^{p,q}}$ is immediate. We remark that $\mathcal{C}$ is translation invariant, then $\mathcal{C}(K)\leq C(n,\alpha,s)$ for any compact set $K\subseteq\mathcal{R}$ with ${\rm diam}(K)\leq 1$. Consequently, we have
\begin{align*}
\|f\|_{M^{p,q}}&=\sup_{K}\frac{\|f\chi_{K}\|_{L^{p,q}}}{\mathcal{C}(K)^{1/q}}\\
&\leq C(n,\alpha,s,p,q)\sup_{{\rm diam}(K)\leq 1}\frac{\|f\chi_{K}\|_{L^{p,q}}}{\mathcal{C}(K)^{1/q}}\\
&\leq C(n,\alpha,s,p,q,r)\sup_{{\rm diam}(K)\leq 1}\frac{\|f\chi_{K}\|_{L^{p,r}}}{\mathcal{C}(K)^{1/q}}\\
&=C(n,\alpha,s,p,q,r)\sup_{{\rm diam}(K)\leq 1}\mathcal{C}(K)^{\frac{1}{r}-\frac{1}{q}}\frac{\|f\chi_{K}\|_{L^{p,r}}}{\mathcal{C}(K)^{1/r}}\\
&\leq C'(n,\alpha,s,p,q,r)\sup_{{\rm diam}(K)\leq 1}\frac{\|f\chi_{K}\|_{L^{p,r}}}{\mathcal{C}(K)^{1/r}}\\
&\leq C'(n,\alpha,s,p,q,r)\|f\|_{M^{p,r}}
\end{align*}
Subsequently, by letting $\varepsilon=q/p$ as in (\ref{sobolev 1}) and (\ref{sobolev 2}), we obtain
\begin{align*}
\|f\|_{M^{p,q}}&=\sup_{K}\frac{\|f\chi_{K}\|_{L^{p,q}}}{\mathcal{C}(K)^{1/q}}\\
&\leq\sup_{K}\frac{\|f\|_{L^{\infty}}\|\chi_{K}\|_{L^{p,q}}}{\mathcal{C}(K)^{1/q}}\\
&=\|f\|_{L^{\infty}}\sup_{K}\frac{|K|^{1/p}}{\mathcal{C}(K)^{1/q}}\\
&\leq C(n,\alpha,s,p,q)\|f\|_{L^{\infty}},
\end{align*}
which finishes the proof.
\end{proof}

\begin{remark}
\rm Note that $|E|\leq C(n,\alpha,s)\mathcal{C}(E)$ holds for $\alpha s\leq n$. For any $1<p<\infty$ and $1<q<\infty$, we have $\|f\chi_{K}\|_{L^{p,q}}\leq\|f\|_{L^{\infty}}|K|^{1/p}\leq C(n,\alpha,s,p)\|f\|_{L^{\infty}}\mathcal{C}(E)^{1/p}$, whence $L^{\infty}\hookrightarrow\mathcal{M}^{p,q}$.
\end{remark}

To prove Theorem \ref{second predual}, we need the following technical lemma.
\begin{lemma}\label{norm switching}
Let $1<p<\infty$ and $1<q<\infty$. For any $f\in\mathcal{M}(\mathcal{R})$, it holds that
\begin{align*}
\|f\|_{M^{p,q}}=\sup_{E}\frac{\|f\chi_{E}\|_{L^{p,q}}}{\mathcal{C}(E)^{1/q}},
\end{align*}
where the supremum is taken over all measurable sets $E\subseteq\mathcal{R}$ with $0<\mathcal{C}(E)<\infty$.
\end{lemma}

\begin{proof}
Denote by 
\begin{align*}
M&=\sup_{E}\frac{\|f\chi_{E}\|_{L^{p,q}}}{\mathcal{C}(E)^{1/q}},\\
M'&=\sup_{G}\frac{\|f\chi_{G}\|_{L^{p,q}}}{\mathcal{C}(G)^{1/q}}
\end{align*}
for simplicity, where the supremum in $M'$ is taken over all open sets $G\subseteq\mathcal{R}$ with $0<\mathcal{C}(G)<\infty$. It is evident that $\|f\|_{M^{p,q}}\leq M$. To show that $M\leq\|f\|_{M^{p,q}}$, we need the auxiliary term $M'$. To this end, let $E\subseteq\mathcal{R}$ be measurable with $0<\mathcal{C}(E)<\infty$. By the outer regularity of $\mathcal{C}$, one can find a sequence $\{G_{n}\}_{n=1}^{\infty}$ of open subsets of $\mathcal{R}$ such that $G_{n}\supseteq E$, $\mathcal{C}(G_{n})<\infty$, and $\mathcal{C}(G_{n})\rightarrow\mathcal{C}(E)$. Then
\begin{align*}
\frac{\|f\chi_{E}\|_{L^{p,q}}}{\mathcal{C}(E)^{1/q}}=\lim_{n\rightarrow\infty}\frac{\|f\chi_{E}\|_{L^{p,q}}}{\mathcal{C}(G_{n})^{1/q}}\leq\sup_{n\in\mathbb{N}}\frac{\|f\chi_{G_{n}}\|_{L^{p,q}}}{\mathcal{C}(G_{n})^{1/q}}\leq M'.
\end{align*}
Taking supremum over all such $E$, we deduce that $M\leq M'$. It remains to show that $M'=\|f\|_{M^{p,q}}$. Let $G\subseteq\mathcal{R}$ be open with $0<\mathcal{C}(G)<\infty$. Since $G$ is open, there exists an increasing sequence $\{K_{n}\}_{n=1}^{\infty}$ of compact subsets of $G$ such that 
\begin{align*}
G=\bigcup_{n\in\mathbb{N}}K_{n}.
\end{align*}
Then $\mathcal{C}(K_{n})\rightarrow\mathcal{C}(G)$ and $f\chi_{K_{n}}\uparrow f\chi_{G}$, whence 
\begin{align*}
\frac{\|f\chi_{G}\|_{L^{p,q}}}{\mathcal{C}(G)^{1/q}}=\lim_{n\rightarrow\infty}\frac{\|f\chi_{K_{n}}\|_{L^{p,q}}}{\mathcal{C}(K_{n})^{1/q}}\leq\|f\|_{M^{p,q}}.
\end{align*}
Taking supremum over all such $G$, one obtains $M'\leq\|f\|_{M^{p,q}}$, which completes the proof.
\end{proof}

\begin{remark}
\rm By re-examining the above proof, one may obtain a similar formula for $\left\|\cdot\right\|_{\mathcal{M}^{p,q}}$, where $1<p<\infty$ and $1<q\leq\infty$.
\end{remark}

\begin{proof}[Proof of Theorem \ref{second predual}]
We will prove only for $(B^{p',q'})^{\ast}\approx M^{p,q}$, the assertion that $(\mathcal{B}^{p',q'})^{\ast}\approx\mathcal{M}^{p,q}$ can be argued in a similar fashion. Let $f\in M^{p,q}$ and $g\in B^{p',q'}$ be given. Decompose $g$ into an arbitrary block decomposition of the form
\begin{align}\label{usual decomposition}
g=\sum_{k}\lambda_{k}b_{k},
\end{align}
as in the definition of $B^{p',q'}$. We may assume the associated $E_{k}$ satisfies $\mathcal{C}(E_{k})>0$, for otherwise the term $b_{k}$ in (\ref{usual decomposition}) can be dropped as $b_{k}$ would vanish on $E_{k}$ almost everywhere. Note that $\mathcal{C}(E_{k})<\infty$ as $E_{k}$ is a bounded set. Then Lemma \ref{norm switching} gives
\begin{align}
|\mathcal{L}_{f}(g)|&=\left|\int_{\mathcal{R}}f(x)g(x)dx\right|\notag\\
&\leq\int_{\mathcal{R}}|f(x)g(x)|dx\notag\\
&\leq\sum_{k}|\lambda_{k}|\int_{E_{k}}|f(x)||b_{k}(x)|dx\notag\\
&\leq C(p,q)\sum_{k}|\lambda_{k}|\|f\chi_{E_{k}}\|_{L^{p,q}}\|b_{k}\|_{L^{p',q'}}\notag\\
&\leq C(p,q)\sum_{k}|\lambda_{k}|\frac{\|f\chi_{E_{k}}\|_{L^{p,q}}}{\mathcal{C}(E_{k})^{1/q}}\mathcal{C}(E_{k})^{\frac{1}{q}}\|b_{k}\|_{L^{p',q'}}\notag\\
&\leq C(p,q)\|f\|_{M^{p,q}}\sum_{k}|\lambda_{k}|.\label{4.1}
\end{align}
Taking infimum to all such block decompositions, one obtains $M^{p,q}\hookrightarrow(B^{p',q'})^{\ast}$. We also infer from (\ref{4.1}) that 
\begin{align*}
\|f\|_{(B^{p',q'})'}=\sup\left\{\int_{\mathcal{R}}|f(x)g(x)|dx:\|g\|_{B^{p',q'}}\leq 1\right\}\leq C(p,q)\|f\|_{M^{p,q}},
\end{align*}
which gives $M^{p,q}\hookrightarrow(B^{p',q'})'$.

Conversely, let $\mathcal{L}\in(B^{p',q'})^{\ast}$ be given. If $g$ is a nonzero function in $L^{p',q'}$ with $\{g\ne 0\}\subseteq E$ for some bounded set, then it is easy to see that $g\in B^{p',q'}$ as we may write that $g=\mathcal{C}(E)^{1/q}\|g\|_{L^{p',q'}}\widetilde{g}$ with $\widetilde{g}=g/(\mathcal{C}(E)^{1/q}\|g\|_{L^{p',q'}})$. This leads to 
\begin{align*}
\|g\|_{B^{p',q'}}\leq\mathcal{C}(E)^{\frac{1}{q}}\|g\|_{L^{p',q'}},
\end{align*}
and hence 
\begin{align*}
|\mathcal{L}(g)|\leq\|\mathcal{L}\|\mathcal{C}(E)^{\frac{1}{q}}\|g\|_{L^{p',q'}}.
\end{align*}
In other words, the restriction functional $\mathcal{L}|_{L^{p',q'}(E)}:L^{p',q'}(E)\rightarrow\mathbb{R}$ is continuous. Using the fact that $L^{p,q}(E)\approx(L^{p',q'}(E))^{\ast}$, there is an $f_{E}\in L^{p,q}(E)$ such that 
\begin{align*}
\mathcal{L}(g)=\int_{E}f_{E}(x)g(x)dx,\quad g\in L^{p',q'}(E).
\end{align*}
By restricting $E_{N}=[-N,N]^{n}$, $N\in\mathbb{N}$, we have
\begin{align*}
\int_{E_{N}}f_{E_{N}}(x)g(x)dx=\int_{E_{N}}f_{E_{N+1}}(x)g(x)dx
\end{align*}
for all continuous functions $g$ on $E_{N}$, which yields $f_{E_{N}}=f_{E_{N+1}}$ almost everywhere on $E_{N}$ by the density result in Proposition \ref{lorentz dense}. Then we may define unambiguously a function $f$ with $f(x)=f_{E_{N}}(x)$ for $x\in E_{N}$. Then 
\begin{align}\label{4.2}
\mathcal{L}(g)=\int_{\mathcal{R}}f(x)g(x)dx
\end{align}
holds for any $g\in L^{p',q'}$ with $\{g\ne 0\}\subseteq E_{N}$. In particular, given a compact set $K\subseteq\mathcal{R}$, say, $K\subseteq E_{N}$, and $g={\rm sgn}(f)\chi_{K}$, the equality (\ref{4.2}) shows that $f$ is integrable on $K$. Since the compact set $K$ is arbitrary, we deduce that $f$ is locally integrable on $\mathcal{R}$. Repeating the proof of \cite[Theorem 1.4.16 (vi)]{GL}, one can show that the function $f$ satisfies 
\begin{align*}
\|f\chi_{K}\chi_{|f|\leq N}\|_{L^{p,q}}\leq C(p,q)\|\mathcal{L}\|\mathcal{C}(K)^{\frac{1}{q}},\quad N=1,2,\ldots.
\end{align*}
Taking $N\rightarrow\infty$, we have $\|f\chi_{K}\|_{L^{p,q}}\leq C(p,q)\|\mathcal{L}\|\mathcal{C}(K)^{1/q}$, which gives 
\begin{align}\label{norm dual}
\|f\|_{M^{p,q}}\leq C(p,q)\|\mathcal{L}\|.
\end{align}
Now we extend the validity of (\ref{4.2}) to all $g\in B^{p',q'}$. As before, let $g$ be as in (\ref{usual decomposition}). Consider the compactly supported functions $g_{n}$, $n=1,2,\ldots$ defined by
\begin{align*}
g_{n}=\sum_{|k|\leq n}\lambda_{k}b_{k}.
\end{align*}
The sequence $\{g_{n}\}_{n=1}^{\infty}$ converges to $g$ in $B^{p',q'}$ as $n\rightarrow\infty$. On the other hand, the functions 
\begin{align*}
h_{n}=\sum_{|k|\leq n}|\lambda_{k}||b_{k}|,\quad n=1,2,\ldots,
\end{align*}
belong to $B^{p',q'}$ with 
\begin{align*}
\|h_{n}\|_{B^{p',q'}}\leq\sum_{|k|\leq n}|\lambda_{k}|.
\end{align*}
Using (\ref{4.1}) with $|f|$ in place of $f$, we have 
\begin{align*}
\int_{\mathcal{R}}|f(x)h_{n}(x)|dx\leq C(p,q)\|f\|_{M^{p,q}}\|h_{n}\|_{B^{p',q'}}\leq C(p,q)\|f\|_{M^{p,q}}\sum_{|k|\leq n}|\lambda_{k}|.
\end{align*}
Taking $n\rightarrow\infty$ and using monotone convergence theorem, one has
\begin{align*}
\int_{\mathcal{R}}|f(x)|\sum_{k}|\lambda_{k}||b_{k}(x)|dx\leq C(p,q)\|f\|\sum_{k}|\lambda_{k}|<\infty,
\end{align*}
which also shows that $fg\in L^{1}$. As a result, using Lebesgue dominated convergence theorem and the continuity of $\mathcal{L}$, we have
\begin{align*}
\mathcal{L}(g)=\lim_{n\rightarrow\infty}\mathcal{L}(g_{n})=\int_{\mathcal{R}}f(x)g_{n}(x)dx=\int_{\mathcal{R}}f(x)g(x)dx,
\end{align*}
which yields (\ref{4.2}) for general $g\in B^{p',q'}$. We conclude that $(B^{p',q'})^{\ast}\hookrightarrow M^{p,q}$. Further, we infer from (\ref{norm dual}) that 
\begin{align*}
\|f\|_{M^{p,q}}\leq C(p,q)\sup\left\{\left|\int_{\mathcal{R}}f(x)g(x)dx\right|:\|g\|_{B^{p',q'}}\leq 1\right\}\leq C(p,q)\|f\|_{(B^{p',q'})'},
\end{align*}
which gives $(B^{p',q'})'\hookrightarrow M^{p,q}$.

Finally, we show that $B^{p',q'}$ is a Banach space. It can be seen from the above argument that if a block decomposition (\ref{usual decomposition}) of $g$ exists, then the sum in (\ref{usual decomposition}) is almost everywhere absolutely convergent. As a consequence, the space $B^{p',q'}$ is topological complete. To finish the proof, we only need to show that if $\|g\|_{B^{p',q'}}=0$, then $g=0$ almost everywhere. To this end, let $f\in C_{0}$ be given. Since $f\in M^{p,q}$, it follows from the duality result above that
\begin{align*}
\int_{\mathcal{R}}|f(x)g(x)|dx\leq C(n,\alpha,s,p,q)\|f\|_{M^{p,q}}\|g\|_{B^{p',q'}}=0.
\end{align*}
Then it is a standard fact in analysis that $g=0$ almost everywhere as $f\in C_{0}$ is arbitrary. The proof is now finished.
\end{proof}

\begin{remark}
\rm For $1<p<\infty$ and $1<q<\infty$, it is worth mentioning that $B^{p,q}$ is solid in the sense that if $f\in B^{p,q}$ and $|g|\leq|f|$, then $g\in B^{p,q}$. To see this, let
\begin{align*}
f=\sum_{k}\lambda_{k}b_{k}
\end{align*}
be a block decomposition. Express $g$ as
\begin{align*}
g=\sum_{k}\lambda_{k}gf^{-1}\chi_{\{f\ne 0\}}b_{k}.
\end{align*}
Then
\begin{align*}
\{gf^{-1}\chi_{\{f\ne 0\}}b_{k}\ne 0\}\subseteq\{b_{k}\ne 0\}\subseteq E_{k}
\end{align*}
and
\begin{align*}
\mathcal{C}(E_{k})^{1/q'}\|gf^{-1}\chi_{\{f\ne 0\}}b_{k}\|_{L^{p,q}}\leq\mathcal{C}(E_{k})^{1/q'}\|b_{k}\|_{L^{p,q}}\leq 1.
\end{align*}
We conclude that 
\begin{align*}
\|g\|_{B^{p,q}}\leq\sum_{k}|\lambda_{k}|.
\end{align*}
Taking infimum over all such block decompositions, one obtains $\|g\|_{B^{p,q}}\leq\|f\|_{B^{p,q}}$. Hence $g\in B^{p,q}$, as claimed.
\end{remark}

\begin{proof}[Proof of Theorem \ref{embed}]
We first show (\ref{first embed}) for $1<q\leq p<\infty$, where its proof is inspired by \cite[Theorem 2.1]{MST}. Define the auxiliary seminorm $\left\|\cdot\right\|$ by
\begin{align}\label{N block}
\|f\|=\inf\left\{\sum_{k}|\lambda_{k}|:f=\sum_{k}\lambda_{k}b_{k}\right\},
\end{align}
where the infimum is taken over all pointwise almost everywhere convergence of infinite sum consisting of $b_{k}\in N^{p,q}$ such that $\|b_{k}\|_{N^{p,q}}\leq 1$. We claim that 
\begin{align}\label{self twist}
C(n,\alpha,s,p,q)^{-1}\left\|\cdot\right\|_{N^{p,q}}\leq\left\|\cdot\right\|\leq C(n,\alpha,s,p,q)\left\|\cdot\right\|_{N^{p,q}}. 
\end{align}
Let $\varepsilon>0$ and write $f=(\|f\|_{N^{p,q}}+\varepsilon)\cdot f/(\|f\|_{N^{p,q}}+\varepsilon)$. One obtains at once that
\begin{align*}
\|f\|\leq\|f\|_{N^{p,q}}+\varepsilon.
\end{align*}
The arbitrariness of $\varepsilon$ gives $\left\|\cdot\right\|\leq\left\|\cdot\right\|_{N^{p,q}}$. For the converse, let $\varepsilon>0$ and $f$ be decomposed as in (\ref{N block}). Then
\begin{align*}
\sum_{k}|\lambda_{k}|\leq\|f\|+\varepsilon.
\end{align*}
As $\|b_{k}\|_{N^{p,q}}\leq 1$, there is an $\omega_{k}\in L^{1}(\mathcal{C})\cap A_{1}^{\rm loc}$ such that $\|\omega_{k}\|_{L^{1}(\mathcal{C})}\leq 1$, $[\omega_{k}]_{A_{1}^{\rm loc}}\leq\mathfrak{c}$, and 
\begin{align*}
\|b_{k}\omega_{k}^{-1/p'}\|_{L^{q,p}}\leq \|b_{k}\|_{N^{p,q}}+\varepsilon\leq 1+\varepsilon.
\end{align*}
Using H\"older's inequality, we have 
\begin{align*}
|f|&\leq\sum_{k}|\lambda_{k}||b_{k}|\\
&=\sum_{k}|\lambda_{k}|\omega_{k}^{\frac{1}{q'}}|b_{k}|\omega_{k}^{-\frac{1}{q'}}\\
&\leq\left(\sum_{k}|\lambda_{k}|\omega_{k}\right)^{\frac{1}{q'}}\left(\sum_{k}|\lambda_{k}||b_{k}|^{q}\omega_{k}^{1-q}\right)^{\frac{1}{q}}.
\end{align*}
Denote by
\begin{align*}
\omega=\frac{\sum_{k}|\lambda_{k}|\omega_{k}}{\sum_{k}|\lambda_{k}|}.
\end{align*}
First we verify that 
\begin{align*}
\|\omega\|_{L^{1}(\mathcal{C})}\leq C(n,\alpha,s)\frac{1}{\sum_{k}|\lambda_{k}|}\sum_{k}|\lambda_{k}|\|\omega_{k}\|_{L^{1}(C)}\leq C(n,\alpha,s)
\end{align*}
by the normability of $L^{1}(\mathcal{C})$. Moreover, we have 
\begin{align*}
{\bf M}^{\rm loc}\omega\leq\frac{1}{\sum_{k}|\lambda_{k}|}\sum_{k}|\lambda_{k}|{\bf M}^{\rm loc}\omega_{k}\leq\mathfrak{c}\cdot\frac{1}{\sum_{k}|\lambda_{k}|}\sum_{k}|\lambda_{k}|\omega_{k}\leq\mathfrak{c}\cdot\omega
\end{align*}
almost everywhere, which shows that $[\omega]_{A_{1}^{\rm loc}}\leq\mathfrak{c}$. Recall the assumption that $1<q\leq p<\infty$. Then $p/q\geq 1$, whence $\left\|\cdot\right\|_{L^{p/q,1}}$ is quasi-additive. Subsequently, it holds that 
\begin{align*}
\|f\|_{N^{p,q}}&\leq\|f\omega^{-1/q'}\|_{L^{p,q}}\\
&=C(n,\alpha,s,q)\left(\sum_{k}|\lambda_{k}|\right)^{\frac{1}{q'}}\left\|\left(\sum_{k}|\lambda_{k}||b_{k}|^{q}\omega_{k}^{1-q}\right)^{\frac{1}{q}}\right\|_{L^{p,q}}\\
&=C(n,\alpha,s,q)\left(\sum_{k}|\lambda_{k}|\right)^{\frac{1}{q'}}\left\|\sum_{k}|\lambda_{k}||b_{k}|^{q}\omega_{k}^{1-q}\right\|_{L^{p/q,1}}^{\frac{1}{q}}\\
&\leq C(n,\alpha,s,p,q)\left(\sum_{k}|\lambda_{k}|\right)^{\frac{1}{q'}}\left[\sum_{k}|\lambda_{k}|\left\||b_{k}|^{q}\omega_{k}^{1-q}\right\|_{L^{p/q,1}}\right]^{\frac{1}{q}}\\
&=C(n,\alpha,s,p,q)\left(\sum_{k}|\lambda_{k}|\right)^{\frac{1}{q'}}\left[\sum_{k}|\lambda_{k}|\|b_{k}\omega_{k}^{-1/q'}\|_{L^{p,q}}^{q}\right]^{\frac{1}{q}}\\
&\leq C(n,\alpha,s,p,q)(1+\varepsilon)\sum_{k}|\lambda_{k}|.
\end{align*}
Taking infimum over all such decompositions (\ref{N block}), the claim (\ref{self twist}) follows by the arbitrariness of $\varepsilon$. 

Let $f\in B^{p,q}$ be given.  Decompose $f$ into an arbitrary block decomposition of the form
\begin{align*}
f=\sum_{k}\lambda_{k}b_{k},
\end{align*}
as in the definition of $B^{p,q}$. Note that $\{b_{k}\ne 0\}\subseteq E_{k}$, where $E_{k}$ is bounded with $\mathcal{C}(E_{k})^{1/q'}\|b_{k}\|_{L^{p,q}}\leq 1$. Let $V^{E_{k}}$ be the nonlinear potential associated to $E_{k}$ with the admissible exponent $\delta$. Then $(V^{E_{k}})^{\delta}/\mathcal{C}(E_{k})\in L^{1}(\mathcal{C})\cap A_{1}^{\rm loc}$, $[(V^{E_{k}})^{\delta}/\mathcal{C}(E_{k})]_{A_{1}^{\rm loc}}\leq\mathfrak{c}$, and $\|(V^{E_{k}})^{\delta}/\mathcal{C}(E_{k})\|_{L^{1}(\mathcal{C})}\leq C(n,\alpha,s)$. Hence
\begin{align*}
\|b_{k}\|_{N^{p,q}}&\leq C(n,\alpha,s,p,q)\left\|b_{k}\left[\frac{(V^{E_{k}})^{\delta}}{\mathcal{C}(E_{k})}\right]^{-1/q'}\right\|_{L^{p,q}}\\
&\leq C(n,\alpha,s,p,q)\mathcal{C}(E_{k})^{1/q'}\|b_{k}\|_{L^{p,q}}\\
&\leq C(n,\alpha,s,p,q)
\end{align*}
by noting that $(V^{E_{k}})^{\delta}\geq 1$ on $E_{k}$ and $b_{k}=0$ on $\mathcal{R}\setminus E_{k}$. Combining (\ref{N block}) with (\ref{self twist}), we deduce that 
\begin{align*}
\|f\|_{N^{p,q}}\leq C(n,\alpha,s,p,q)\|f\|_{B^{p,q}},
\end{align*}
which gives $B^{p,q}\hookrightarrow N^{q,p}$ when $1<q\leq p<\infty$.

It remains to show (\ref{second embed}) for $1<p\leq q<\infty$. We may assume for $p<q$ as the case where $p=q$ reduces to the result in \cite{OP}. Let $\varepsilon>0$ and $f\in N^{q,p}$ be given. There is an $\omega\in L^{1}(\mathcal{C})$ such that $\|\omega\|_{L^{1}(\mathcal{C})}\leq 1$ and 
\begin{align*}
\|f\omega^{-1/q'}\|_{L^{p,q}}\leq\|f\|_{N^{p,q}}+\varepsilon.
\end{align*}
Note that 
\begin{align*}
\sum_{k\in\mathbb{Z}}2^{k}\mathcal{C}(\{2^{k-1}<\omega\leq 2^{k}\})&=\frac{1}{4}\sum_{k\in\mathbb{Z}}\int_{2^{k-2}}^{2^{k-1}}\mathcal{C}(\{2^{k-1}<\omega\leq 2^{k}\})dt\\
&\leq\frac{1}{4}\sum_{k\in\mathbb{Z}}\int_{2^{k-2}}^{2^{k-1}}\mathcal{C}(\{\omega>t\})dt\\
&=\frac{1}{4}\|\omega\|_{L^{1}(\mathcal{C})}\\
&\leq\frac{1}{4}.
\end{align*}
Let $E_{k}=\{2^{k-1}<\omega\leq 2^{k}\}$ and $D_{l}=\{l-1\leq|x|<l\}$, $k\in\mathbb{Z}$, $l\in\mathbb{N}$. Since $\omega(x)<\infty$ $\mathcal{C}$-quasi-everywhere, we can express $f$ as
\begin{align*}
f=\sum_{k,l}f\chi_{E_{k}\cap D_{l}}=\sum_{k,l}\lambda_{k,l}b_{k,l}
\end{align*}
almost everywhere with 
\begin{align*}
\lambda_{k,l}&=\|f\chi_{E_{k}\cap D_{l}}\|_{L^{p,q}}\mathcal{C}(E_{k}\cap D_{l})^{1/q'},\\
b_{k,l}&=\|f\chi_{E_{k}\cap D_{l}}\|_{L^{p,q}}^{-1}\mathcal{C}(E_{k}\cap D_{l})^{-1/q'}f\chi_{E_{k}\cap D_{l}},
\end{align*}
where we have used the convention that $b_{k,l}=0$ whenever $f=0$ almost everywhere on $E_{k}\cap D_{l}$ in view of (\ref{sobolev 1}) and (\ref{sobolev 2}). It is immediate that 
\begin{align*}
\mathcal{C}(E_{k}\cap D_{l})^{1/q'}\|b_{k,l}\|_{L^{p,q}}=1.
\end{align*}
Subsequently, it follows that 
\begin{align*}
\sum_{k,l}\lambda_{k,l}&=\sum_{k,l}\|f\chi_{E_{k}\cap D_{l}}\|_{L^{p,q}}\mathcal{C}(E_{k}\cap D_{l})^{1/q'}\\
&=\sum_{k,l}\|f\omega^{-1/q'}\omega^{1/q'}\chi_{E_{k}\cap D_{l}}\|_{L^{p,q}}\mathcal{C}(E_{k}\cap D_{l})^{1/q'}\\
&\leq\sum_{k,l}\|f\omega^{-1/q'}\chi_{E_{k}\cap D_{l}}\|_{L^{p,q}}2^{k/q'}\mathcal{C}(E_{k}\cap D_{l})^{1/q'}\\
&\leq\left(\sum_{k,l}\|f\omega^{-1/q'}\chi_{E_{k}\cap D_{l}}\|_{L^{p,q}}^{q}\right)^{\frac{1}{q}}\left(\sum_{k,l}2^{k}\mathcal{C}(E_{k}\cap D_{l})\right)^{\frac{1}{q'}}.
\end{align*}
Since $1<p<q<\infty$, repeating the proof of Theorem \ref{first predual} with $r'=q'$, we have
\begin{align*}
\left(\sum_{k,l}\|f\omega^{-1/q'}\chi_{E_{k}\cap D_{l}}\|_{L^{p,q}}^{q}\right)^{\frac{1}{q}}&\leq C(p,q)\left\|\left(\sum_{k,l}(f\omega^{-1/q'})^{q}\chi_{E_{k}\cap D_{l}}\right)^{\frac{1}{q}}\right\|_{L^{p,q}}\\
&=C(p,q)\|f\omega^{-1/q'}\|_{L^{p,q}}.
\end{align*}
Moreover, using the quasi-additivity of $\mathcal{C}$ (see \cite[Theorem 4]{AD}), we obtain
\begin{align*}
\sum_{k,l}2^{k}\mathcal{C}(E_{k}\cap D_{l})\leq C(n,\alpha,s)\sum_{k}2^{k}\mathcal{C}(E_{k})\leq C'(n,\alpha,s).
\end{align*}
We conclude that 
\begin{align*}
\sum_{k,l}\lambda_{k,l}\leq C(n,\alpha,s,p,q)\|f\omega^{-1/q'}\|_{L^{p,q}}\leq C(n,\alpha,s,p,q)(\|f\|_{N^{p,q}}+\varepsilon),
\end{align*}
and hence 
\begin{align*}
\|f\|_{B^{p,q}}\leq C(n,\alpha,s,p,q)(\|f\|_{N^{p,q}}+\varepsilon).
\end{align*}
The arbitrariness of $\varepsilon$ leads to $N^{p,q}\hookrightarrow B^{p,q}$ for $1<p<q<\infty$. The proof is now complete.
\end{proof}

\begin{proof}[Proof of Proposition \ref{fatou property}]
We first prove for the weak Fatou's property of $N^{p,q}$. To this end, let $\{f_{n}\}_{n=1}^{\infty}\subseteq N^{p,q}$ be an increasing sequence of nonnegative functions. Denote by 
\begin{align*}
f=\sup\limits_{n\in\mathbb{N}}f_{n}
\end{align*}
and
\begin{align*}
M=\sup_{n\in\mathbb{N}}\|f_{n}\|_{N^{p,q}}<\infty
\end{align*}
for simplicity. Let $\varepsilon>0$ be given. For each $n\in\mathbb{N}$, there is an $\omega_{n}$ as in the definition of $N^{p,q}$ such that 
\begin{align*}
\|f_{n}\omega_{n}^{-1/q'}\|_{N^{p,q}}<\|f_{n}\|_{L^{p,q}}+\varepsilon.
\end{align*}
Since $\|\omega_{n}\|_{L^{1}(\mathcal{C})}\leq 1$ for all $n\in\mathbb{N}$ and $L^{1}(\mathcal{C})\hookrightarrow L^{1}$, we can use K\'omlos theorem to obtain a subsequence, still denoted by $\{\omega_{n}\}_{n=1}^{\infty}$, and such that
\begin{align*}
\sigma_{n}=\frac{\omega_{1}+\cdots+\omega_{n}}{n}\rightarrow\overline{\omega}
\end{align*}
almost everywhere for some nonnegative measurable function $\overline{\omega}$ (see \cite[Theorem 3.1]{DL}). Define
\begin{align*}
\omega(x)=\begin{cases}
\lim_{n\rightarrow\infty}\sigma_{n}(x),\quad&\text{if the limit exists},\\
0,\quad&\text{otherwise}.
\end{cases}
\end{align*}
Then $\omega\leq\liminf\limits_{n\rightarrow\infty}\sigma_{n}$ and $\omega=\overline{\omega}$ almost everywhere. Further, we have 
\begin{align*}
\|\omega\|_{L^{1}(\mathcal{C})}&\leq\|\liminf_{n\rightarrow\infty}\sigma_{n}\|_{L^{1}(\mathcal{C})}\\
&\leq\liminf_{n\rightarrow\infty}\|\sigma_{n}\|_{L^{1}(\mathcal{C})}\\
&\leq C(n,\alpha,s)\liminf_{n\rightarrow\infty}\frac{1}{n}\sum_{k=1}^{n}\|\omega_{k}\|_{L^{1}(\mathcal{C})}\\
&\leq C(n,\alpha,s).
\end{align*}
Moreover, it holds that 
\begin{align*}
{\bf M}^{\rm loc}\omega={\bf M}^{\rm loc}\overline{\omega}\leq\liminf_{n\rightarrow\infty}\frac{1}{n}\sum_{k=1}^{n}{\bf M}^{\rm loc}\omega_{k}\leq\mathfrak{c}\cdot\omega
\end{align*}
almost everywhere, which yields $\omega\in A_{1}^{\rm loc}$ with $[\omega]_{A_{1}^{\rm loc}}\leq\mathfrak{c}$. In other words, $\omega$ satisfies the weight conditions in $N^{p,q}$. For fixed $k\in\mathbb{N}$, we have 
\begin{align*}
\|f_{k}\omega^{-1/q'}\|_{L^{p,q}}&=\|f_{k}\overline{\omega}^{-1/q'}\|_{L^{p,q}}\\
&\leq\liminf_{n\rightarrow\infty}\left\|f_{k}\left(\frac{1}{n}\sum_{i=k}^{k+n-1}\omega_{i}\right)^{-1/q'}\right\|_{L^{p,q}}\\
&\leq\liminf_{n\rightarrow\infty}\left\|\frac{1}{n}\sum_{i=k}^{k+n-1}f_{k}\omega_{i}^{-1/q'}\right\|_{L^{p,q}}\\
&\leq C(n,\alpha,s,p,q)\liminf_{n\rightarrow\infty}\frac{1}{n}\sum_{i=k}^{k+n-1}\|f_{k}\omega_{i}^{-1/q'}\|_{L^{p,q}}\\
&\leq C(n,\alpha,s,p,q)\liminf_{n\rightarrow\infty}\frac{1}{n}\sum_{i=k}^{k+n-1}\|f_{i}\omega_{i}^{-1/q'}\|_{L^{p,q}}\\
&\leq C(n,\alpha,s,p,q)(M+\varepsilon),
\end{align*}
where we have used Jensen's inequality for the convex function $u\rightarrow u^{-1/q'}$, $u>0$, and also the fact that $f_{k}\leq f_{i}$ for $k\leq i$. Taking $k\rightarrow\infty$, we obtain
\begin{align*}
\|f\|_{N^{p,q}}\leq C(n,\alpha,s,p,q)\|f\omega^{-1/q'}\|_{L^{p,q}}\leq C(n,\alpha,s,p,q)(M+\varepsilon).
\end{align*}
Since $\varepsilon$ is arbitrary, it follows that $\|f\|_{N^{p,q}}\leq C(n,\alpha,s,p,q)M$, which completes the proof for weak Fatou's property.

It remains to show that $\|f\|_{N^{p,q}}=0$ entails $f=0$ almost everywhere. For each $n\in\mathbb{N}$, choose an $\omega_{n}$ as in the definition of $N^{p,q}$ such that 
\begin{align*}
\|f\omega_{n}^{-1/q'}\|_{L^{p,q}}<2^{-n}.
\end{align*}
By defining $\omega$ as before, we obtain
\begin{align*}
\|f\omega^{-1/q'}\|_{L^{p,q}}&\leq\liminf_{n\rightarrow\infty}\left\|f\left(\frac{1}{n}\sum_{k=1}^{n}\omega_{k}\right)^{-1/q'}\right\|_{L^{p,q}}\\
&\leq\liminf_{n\rightarrow\infty}\left\|\frac{1}{n}\sum_{k=1}^{n}f\omega_{k}^{-1/q'}\right\|_{L^{p,q}}\\
&\leq C(n,\alpha,s,p,q)\liminf_{n\rightarrow\infty}\frac{1}{n}\sum_{k=1}^{n}\|f\omega_{k}^{-1/q'}\|_{L^{p,q}}\\
&\leq C(n,\alpha,s,p,q)\liminf_{n\rightarrow\infty}\frac{1}{n}\sum_{k=1}^{\infty}\|f\omega_{k}^{-1/q'}\|_{L^{p,q}}\\
&\leq C(n,\alpha,s,p,q)\liminf_{n\rightarrow\infty}\frac{1}{n}\sum_{k=1}^{\infty}\frac{1}{2^{k}}\\
&=0.
\end{align*}
Therefore, we have $f\omega^{-1/q'}=0$ almost everywhere. Since $\omega$ is a weight, we deduce that $f=0$ almost everywhere. For $1<q\leq p<\infty$, it is shown in (\ref{N block}) that $\left\|\cdot\right\|_{N^{p,q}}\approx\left\|\cdot\right\|$, while the space $N^{p,q}$ endowed with $\left\|\cdot\right\|$ is clearly complete by the almost everywhere absolutely convergent series defined in $\left\|\cdot\right\|$. We conclude that $N^{p,q}$ is a complete normable space, which finishes the proof.
\end{proof}

\begin{proof}[Proof of Proposition \ref{density}]
We first show the density of $C_{0}$ in $B^{p,q}$. Let $f\in B^{p,q}$ and $\varepsilon>0$ be given. Then there is a block decomposition of $f$ that 
\begin{align*}
f=\sum_{k}\lambda_{k}b_{k}
\end{align*}
and 
\begin{align*}
\sum_{k}|\lambda_{k}|<\|f\|_{B^{p,q}}+1<\infty
\end{align*}
as in the definition of $B^{p,q}$. As shown in the proof of Theorem \ref{second predual}, the block decomposition is almost everywhere absolutely convergent. Then
\begin{align*}
\left\|f-\sum_{|k|\leq n}\lambda_{k}b_{k}\right\|_{B^{p,q}}=\left\|\sum_{|k|>n}\lambda_{k}b_{k}\right\|_{B^{p,q}}\leq\sum_{|k|>n}|\lambda_{k}|\rightarrow 0,\quad n\rightarrow\infty.
\end{align*}
Then there is an $n\in\mathbb{N}$ such that 
\begin{align*}
\left\|f-\sum_{|k|\leq n}\lambda_{k}b_{k}\right\|_{B^{p,q}}<\varepsilon.
\end{align*}
Using Proposition \ref{lorentz dense}, for each $k=1,2,\ldots,n$, there is a $\varphi_{k}\in C_{0}$ such that 
\begin{align*}
\|b_{k}-\varphi_{k}\|_{L^{p,q}}<\left(1+\sum_{k}|\lambda_{k}|\right)^{-1}\varepsilon.
\end{align*}
Then the function $\varphi$ defined by
\begin{align*}
\varphi=\sum_{|k|\leq n}\lambda_{k}\varphi_{k}
\end{align*}
belongs to $C_{0}$ such that 
\begin{align*}
\|f-\varphi\|_{B^{p,q}}&\leq\left\|f-\sum_{|k|\leq n}\lambda_{k}b_{k}\right\|_{B^{p,q}}+\left\|\sum_{|k|\leq n}\lambda_{k}b_{k}-\varphi\right\|_{B^{p,q}}\\
&<\varepsilon+\sum_{|k|\leq n}|\lambda_{k}|\|f-\varphi_{k}\|_{B^{p,q}}\\
&\leq\varepsilon+\left(\sum_{|k|\leq n}|\lambda_{k}|\right)\left(1+\sum_{k}|\lambda_{k}|\right)^{-1}\varepsilon\\
&\leq 2\varepsilon,
\end{align*}
which justifies the density claim.

For the density of $C_{0}$ in $N^{p,q}$, we proceed as follows. Let $f\in N^{p,q}$ and $\varepsilon>0$. Then there is an $\omega$ as in the definition of $N^{p,q}$ such that 
\begin{align*}
\|f\omega^{-1/q'}\|_{L^{p,q}}<\|f\|_{N^{p,q}}+1<\infty.
\end{align*}
For each $n\in\mathbb{N}$, define $f_{n}=\min(f,n)\chi_{\{|x|\leq n\}}$. Then $f_{n}\uparrow f$ pointwise everywhere. We claim that 
\begin{align}\label{lebesgue}
\|f-f_{n}\|_{N^{p,q}}\rightarrow 0,\quad n\rightarrow\infty.
\end{align}
To this end, we may first assume that $|f|<\infty$ almost everywhere as $f\omega^{-1/q'}\in L^{p,q}$ and $\omega$ is a weight. Then $f-f_{n}\downarrow 0$ almost everywhere. For fixed $0<t<\infty$, we have 
\begin{align*}
|\{(f-f_{n})\omega^{-1/q'}>t\}|\downarrow 0
\end{align*}
by Lebesgue dominated convergence theorem as $0\leq(f-f_{n})\omega^{-1/q'}\leq f\omega^{-1/q'}$ and $|\{f\omega^{-1/q'}>t\}|<\infty$. As a consequence, using Lebesgue dominated convergence theorem again, one has 
\begin{align*}
\|(f-f_{n})\omega^{-1/q'}\|_{L^{p,q}}^{q}=p\int_{0}^{\infty}|\{(f-f_{n})\omega^{-1/q'}>t\}|^{\frac{q}{p}}t^{q}\frac{dt}{t}\rightarrow 0
\end{align*}
as $|\{(f-f_{n})\omega^{-1/q'}>t\}|^{q/p}\leq|\{f\omega^{-1/q'}>t\}|^{q/p}$ and $f\omega^{-1/q'}\in L^{p,q}$. Then 
\begin{align*}
\|f-f_{n}\|_{N^{p,q}}\leq\|(f-f_{n})\omega^{-1/q'}\|_{L^{p,q}}\rightarrow 0,\quad n\rightarrow\infty,
\end{align*}
which yields (\ref{lebesgue}). As a result, there is an $n\in\mathbb{N}$ such that $\|f-f_{n}\|_{N^{p,q}}<\varepsilon$. Using Proposition \ref{lorentz dense}, we can choose a $\varphi\in C_{0}$ such that $\|f_{n}-\varphi\|_{L^{p,q}}<\varepsilon$. Let $\omega$ be the nonlinear potential associated with the set $\{|x|\leq M\}$, where $M>0$ is such that $\{f_{n}\ne 0\}\cup\{\varphi\ne 0\}\subseteq\{|x|\leq M\}$. Then $\omega\geq\chi_{\{|x|\leq M\}}$ and hence 
\begin{align*}
\|f-\varphi\|_{N^{p,q}}&\leq C(n,\alpha,s,p,q)(\|f-f_{n}\|_{N^{p,q}}+\|f_{n}-\varphi\|_{N^{p,q}})\\
&<C(n,\alpha,s,p,q)(\varepsilon+\|(f_{n}-\varphi)\omega^{-1/q'}\|_{L^{p,q}})\\
&<C'(n,\alpha,s,p,q)\varepsilon,
\end{align*}
which finishes the proof.
\end{proof}

\begin{remark}
\rm One can easily modify the above proof to obtain the density of $C_{0}$ in $\mathcal{B}^{p,q}$ for $1<p<\infty$ and $0<q\leq 1$.  
\end{remark}

\begin{proof}[Proof of Theorem \ref{char M}]
Since $1<q\leq p<\infty$, we have $1<p'\leq q'<\infty$. Using (\ref{second embed}) that $N^{p',q'}\hookrightarrow B^{p',q'}$, $(L^{p',q'})'\approx L^{p,q}$, and the predual that $(B^{p',q'})^{\ast}\approx M^{p,q}$, for such weights $\omega$ in (\ref{M char}), we have
\begin{align*}
\|f\omega^{\frac{1}{q}}\|_{L^{p,q}}&\leq C(p,q)\|f\omega^{\frac{1}{q}}\|_{(L^{p',q'})'}\\
&=C(p,q)\sup\left\{\int_{\mathcal{R}}|f(x)\omega(x)^{\frac{1}{q}}g(x)|dx:\|g\|_{L^{p',q'}}\leq 1\right\}\\
&\leq C(n,\alpha,s,p,q)\sup\{\|f\|_{M^{p,q}}\|g\omega^{\frac{1}{q}}\|_{B^{p',q'}}:\|g\|_{L^{p',q'}}\leq 1\}\\
&\leq C'(n,\alpha,s,p,q)\sup\{\|f\|_{M^{p,q}}\|g\omega^{\frac{1}{q}}\|_{N^{p',q'}}:\|g\|_{L^{p',q'}}\leq 1\}\\
&\leq C'(n,\alpha,s,p,q)\sup\{\|f\|_{M^{p,q}}\|g\omega^{\frac{1}{q}}\omega^{-\frac{1}{q}}\|_{L^{p',q'}}:\|g\|_{L^{p',q'}}\leq 1\}\\
&\leq C'(n,\alpha,s,p,q)\|f\|_{M^{p,q}}.
\end{align*}
Taking supremum over all such $\omega$, we obtain 
\begin{align*}
M\leq C(n,\alpha,s,p,q)\|f\|_{M^{p,q}}.
\end{align*}

For the other direction of the estimate, we let $K\subseteq\mathcal{R}$ be a compact set. Consider the nonlinear potential $V^{K}$ with the admissible exponent $\delta$. Then $(V^{K})^{\delta}\geq\chi_{K}$, $(V^{K})^{\delta}/\mathcal{C}(K)\in L^{1}(\mathcal{C})\cap A_{1}^{\rm loc}$, $[(V^{K})^{\delta}/\mathcal{C}(K)]_{A_{1}^{\rm loc}}\leq\mathfrak{c}$, and 
\begin{align*}
\|(V^{K})^{\delta}/\mathcal{C}(K)\|_{L^{1}(\mathcal{C})}\leq C(n,\alpha,s).
\end{align*}
Subsequently, we have 
\begin{align*}
\frac{\|f\chi_{K}\|_{L^{p,q}}}{\mathcal{C}(K)^{1/q}}=\left\|f\left[\frac{\chi_{K}}{\mathcal{C}(K)}\right]^{\frac{1}{q}}\right\|_{L^{p,q}}\leq\left\|f\left[\frac{(V^{K})^{\delta}}{\mathcal{C}(K)}\right]^{\frac{1}{q}}\right\|_{L^{p,q}}\leq C(n,\alpha,s,p,q)M.
\end{align*}
Taking supremum over all such $K$, we obtain the result.
\end{proof}

\begin{proof}[Proof of Theorem \ref{dual N}]
We first show (\ref{first predual N}) for $1<p<\infty$ and $1<q<\infty$. To this end, let $\mathcal{L}\in(N^{p',q'})^{\ast}$ be given. If $g$ is a nonzero function in $L^{p',q'}$ with $\{g\ne 0\}\subseteq E$ for some bounded set with $\mathcal{C}(E)>0$, then by letting $\omega$ the nonlinear potential associated to $E$, one has
\begin{align*}
\|g\|_{N^{p',q'}}\leq\|g\omega^{-1/q}\|_{L^{p',q'}}\leq C(n,\alpha,s,p,q)\mathcal{C}(E)^{\frac{1}{q}}\|g\|_{L^{p',q'}},
\end{align*}
and hence 
\begin{align*}
|\mathcal{L}(g)|\leq C(n,\alpha,s,p,q)\|\mathcal{L}\|\mathcal{C}(E)^{\frac{1}{q}}\|g\|_{L^{p',q'}}.
\end{align*}
As in the proof of Theorem \ref{second predual}, there is a locally integrable function $f$ on $\mathcal{R}$ with 
\begin{align}\label{second to use}
\mathcal{L}(g)=\int_{\mathcal{R}}f(x)g(x)dx
\end{align}
holds for any $g\in L^{p',q'}$ with $\{g\ne 0\}\subseteq E_{N}$. As shown in the proof therein, the function $f$ satisfies
\begin{align*}
\|f\chi_{K}\chi_{|f|\leq N}\|_{L^{p,q}}\leq C(n,\alpha,s,p,q)\|\mathcal{L}\|\mathcal{C}(K)^{\frac{1}{q}},\quad N=1,2,\ldots,
\end{align*}
and hence
\begin{align*}
\|f\|_{M^{p,q}}\leq C(n,\alpha,s,p,q)\|\mathcal{L}\|.
\end{align*}
To extend the validity of (\ref{second to use}) to all $g\in N^{p',q'}$, we consider the functions $g_{n}$ defined by
\begin{align*}
g_{n}=\min(g,n)\chi_{|x|\leq n},\quad n\in\mathbb{N}.
\end{align*}
As shown in (\ref{M char}), the sequence $\{g_{n}\}_{n=1}^{\infty}$ converges to $g$ in $N^{p',q'}$ as $n\rightarrow\infty$. On the other hand, we have 
\begin{align*}
\int_{\mathcal{R}}|f(x)g_{n}(x)|dx&\leq C(n,\alpha,s,p,q)\|f\|_{M^{p,q}}\|g_{n}\|_{N^{p',q'}}\\
&\leq C(n,\alpha,s,p,q)\|f\|_{M^{p,q}}\|g\|_{N^{p',q'}}.
\end{align*}
Taking $n\rightarrow\infty$ and using monotone convergence theorem, one has
\begin{align*}
\int_{\mathcal{R}}|f(x)g(x)|dx\leq C(n,\alpha,s,p,q)\|f\|_{M^{p,q}}\|g\|_{N^{p',q'}}<\infty,
\end{align*}
which also shows that $fg\in L^{1}(\mathcal{R})$. As a consequence, using Lebesgue dominated convergence theorem and the continuity of $\mathcal{L}$, we have
\begin{align*}
\mathcal{L}(g)=\lim_{n\rightarrow\infty}\mathcal{L}(g_{n})=\int_{\mathcal{R}}f(x)g_{n}(x)dx=\int_{\mathcal{R}}f(x)g(x)dx,
\end{align*}
which yields (\ref{second to use}) for general $g\in N^{p',q'}$. We conclude that $(N^{p',q'})^{\ast}\hookrightarrow M^{p,q}$ and also $(N^{p',q'})'\hookrightarrow M^{p,q}$, hence (\ref{first predual N}) follows.

Now we prove (\ref{predual}) for $1<q\leq p<\infty$. If suffices to show for $M^{p,q}\hookrightarrow(N^{p',q'})^{\ast}$, since combining this with (\ref{first predual N}) would yield $(N^{p',q'})^{\ast}\approx M^{p,q}$ and  $M^{p,q}\hookrightarrow(N^{p',q'})'$. Let $f\in M^{p,q}$ and $g\in N^{p',q'}$ be given. Since $1<p'\leq q'<\infty$, using (\ref{second embed}) in Theorem \ref{embed}, one has 
\begin{align*}
|\mathcal{L}_{f}(g)|&\leq\int_{\mathcal{R}}|f(x)g(x)|dx\\
&\leq C(n,\alpha,s,p,q)\|f\|_{M^{p,q}}\|g\|_{B^{p',q'}}\\
&\leq C'(n,\alpha,s,p,q)\|f\|_{M^{p,q}}\|g\|_{N^{p',q'}},
\end{align*}
which shows that $\mathcal{L}_{f}\in(N^{p',q'})^{\ast}$. Thus $M^{p,q}\hookrightarrow(N^{p',q'})^{\ast}$ follows. 

On the other hand, using Theorem \ref{second predual}, we have $(B^{p',q'})''\approx(M^{p,q})'$. Then the chain (\ref{chain}) follows by (\ref{second embed}) and the fact that $B^{p',q'}\hookrightarrow(B^{p',q'})''$, which completes the proof.
\end{proof}

\begin{proof}[Proof of Theorem \ref{maximal}]
We first justify the boundedness of ${\bf M}^{\rm loc}$ on $M^{p,q}$. Since $A_{1}^{\rm loc}\subseteq A_{q}^{\rm loc}$ (see the definition of $A_{q}^{\rm loc}$ in \cite[Section 2]{OK4}), for any weight $\omega$ as in (\ref{M char}) and $f\in\mathcal{M}(\mathcal{R})$, we have 
\begin{align*}
\|({\bf M}^{\rm loc}f)\omega^{\frac{1}{q}}\|_{L^{q}}&=\|({\bf M}^{\rm loc}f)^{q}\omega\|_{L^{1}}^{1/q}\\
&\leq[C(n,q,[\omega]_{A_{q}^{\rm loc}})\||f|^{q}\omega\|_{L^{1}}]^{1/q}\\
&\leq C(n,\alpha,s,q)\|f\omega^{\frac{1}{q}}\|_{L^{q}},
\end{align*}
as $[\omega]_{A_{q}^{\rm loc}}\leq[\omega]_{A_{1}^{\rm loc}}\leq\mathfrak{c}=C(n,\alpha,s)$ and $C(n,q,\cdot)$ is an increasing function (see \cite[Theorem 2.3.7]{OK4}). On the other hand, since $\omega\in A_{1}^{\rm loc}$ and $[\omega]_{A_{1}^{\rm loc}}\leq\mathfrak{c}$, we have $\omega^{1+\gamma}\in A_{1}^{\rm loc}$, where $\gamma$ can be chosen depending only on $\mathfrak{c}$, which in turn depends on $n,\alpha,s$ (see \cite[Proof of Theorem 2.4.2 and Remark 2.4.3]{OK4}). Then
\begin{align*}
\|({\bf M}^{\rm loc}f)\omega^{\frac{1}{q}}\|_{L^{(1+\gamma)q}}&=\|({\bf M}^{\rm loc}f)^{(1+\gamma)q}\omega^{1+\gamma}\|_{L^{1}}^{1/(1+\gamma)q}\\
&\leq C(n,\alpha,s,q)\||f|^{(1+\gamma)q}\omega^{1+\gamma}\|_{L^{1}}^{1/(1+\gamma)q}\\
&=C(n,\alpha,s,q)\|f\omega^{\frac{1}{q}}\|_{L^{(1+\gamma)q}}.
\end{align*}
Subsequently, if we let 
\begin{align*}
Tf=({\bf M}^{\rm loc}(f\omega^{-\frac{1}{q}}))\omega^{\frac{1}{q}},\quad f\in\mathcal{M}(\mathcal{R}), 
\end{align*}
then we have the boundedness that 
\begin{align*}
&\|Tf\|_{L^{q}}\leq C(n,\alpha,s,q)\|f\|_{L^{q}},\\
&\|Tf\|_{L^{(1+\gamma)q}}\leq C(n,\alpha,s,q)\|f\|_{L^{(1+\gamma)q}}.
\end{align*}
Assume that $q\leq p<(1+\gamma)q$. Using the Marcinkiewicz interpolation theorem of Lorentz type (see \cite{HR} or \cite[Theorem 4.13]{BS}), we have 
\begin{align*}
\|Tf\|_{L^{p,q}}\leq C(n,\alpha,s,p,q)\|f\|_{L^{p,q}}.
\end{align*}
Replacing $f$ in the above estimate with $f\omega^{1/q}$, we obtain
\begin{align*}
\|({\bf M}^{\rm loc}f)\omega^{\frac{1}{q}}\|_{L^{p,q}}\leq C(n,\alpha,s,p,q)\|f\omega^{\frac{1}{q}}\|_{L^{p,q}}.
\end{align*}
Taking supremum over all such $\omega$, the estimate (\ref{boundedness 1}) follows by Theorem \ref{M char}.

It remains to prove the boundedness of ${\bf M}^{\rm loc}$ on $N^{p,q}$. Let $f\in\mathcal{M}(\mathcal{R})$ and $\omega$ be an arbitrary weight as in the definition of $\left\|\cdot\right\|_{N^{p,q}}$. Note that $\omega\in A_{1}^{\rm loc}\subseteq A_{q'}^{\rm loc}$ entails $\omega^{1-q}\in A_{q}^{\rm loc}$ with 
\begin{align*}
[\omega^{1-q}]_{A_{q}^{\rm loc}}=[\omega]_{A_{q'}^{\rm loc}}^{q-1}\leq[\omega]_{A_{1}^{\rm loc}}^{q-1}\leq\mathfrak{c}^{q-1}.
\end{align*}
Then
\begin{align*}
\|({\bf M}^{\rm loc}f)\omega^{-1/q'}\|_{L^{q}}&=\|({\bf M}f)^{q}\omega^{1-q}\|_{L^{1}}^{1/q}\\
&\leq[C(n,\alpha,s,[\omega]_{A_{q}^{\rm loc}})\||f|^{q}\omega^{1-q}\|_{L^{1}}]^{1/q}\\
&\leq C(n,\alpha,s,q)\|f\omega^{-1/q'}\|_{L^{q}}.
\end{align*}
On the other hand, since $\overline{\omega}=\omega^{1-q}\in A_{q}^{\rm loc}$ satisfies $[\overline{\omega}]_{A_{q}^{\rm loc}}\leq\mathfrak{c}^{q-1}$, we have 
\begin{align*}
[\overline{\omega}]_{A_{r}^{\rm loc}}\leq C(n,\alpha,s,q)[\overline{\omega}]_{A_{q}}\leq C'(n,\alpha,s,q)
\end{align*}
with $r=(q+\overline{\gamma})/(1+\overline{\gamma})$ for some $\overline{\gamma}$ depending only on $n,\alpha,s$ (see the proof of \cite[Corollary 2.4.6]{OK4}). As $r/q<1$, we have $\overline{\omega}^{r/q}\in A_{r}^{\rm loc}$ with
\begin{align*}
[\overline{\omega}^{r/q}]_{A_{r}^{\rm loc}}\leq[\overline{\omega}]_{A_{r}^{\rm loc}}^{r/q}\leq C(n,\alpha,s,q)
\end{align*}
(see \cite[Lemma 2.1.4]{OK4}). Then
\begin{align*}
\|({\bf M}^{\rm loc}f)\omega^{-1/q'}\|_{L^{r}}&=\|({\bf M}^{\rm loc}f)^{r}\omega^{-r/q'}\|_{L^{1}}^{1/r}\\
&=\|({\bf M}^{\rm loc}f)^{r}\overline{\omega}^{r/q}\|_{L^{1}}^{1/r}\\
&\leq[C(n,\alpha,s,[\overline{\omega}^{r/q}]_{A_{r}^{\rm loc}})\||f|^{r}\overline{\omega}^{r/q}\|_{L^{1}}]^{1/r}\\
&\leq C(n,\alpha,s,q)\|f\overline{\omega}^{1/q}\|_{L^{r}}\\
&=C(n,\alpha,s,q)\|f\omega^{-1/q'}\|_{L^{r}}.
\end{align*}
Therefore, by defining the operator $T$ that 
\begin{align*}
Tf=[{\bf M}^{\rm loc}(f\omega^{1/q'})]\omega^{-1/q'},
\end{align*}
we have
\begin{align*}
\|Tf\|_{L^{q}}&\leq C(n,\alpha,s,q)\|f\|_{L^{q}},\\
\|Tf\|_{L^{r}}&\leq C(n,\alpha,s,q)\|f\|_{L^{r}}.
\end{align*}
If $1<r<p\leq q$, then using the same interpolation theorem again, we obtain
\begin{align*}
\|Tf\|_{L^{p,q}}\leq C(n,\alpha,s,p,q)\|f\|_{L^{p,q}}.
\end{align*}
By replacing $f$ with $f\omega^{-1/q'}$, the estimate (\ref{boundedness 2}) follows by the definition of $\left\|\cdot\right\|_{N^{p,q}}$. The proof is now complete.
\end{proof}

\begin{proof}[Proof of Theorem \ref{weak dual}]
The proof is very similar to that of Theorem \ref{second predual}, and we will follow the notations therein. With such a block decomposition of $g$ and the canonical functional $\mathcal{L}_{f}$, we have
\begin{align*}
|\mathcal{L}_{f}(g)|&\leq C(p,q)\sum_{k}|\lambda_{k}|\|f\chi_{E_{k}}\|_{L^{p,\infty}}\|b_{k}\|_{L^{p',q}}\\
&\leq C(p,q)\|f\|_{\mathcal{M}^{p,\infty}}\sum_{k}|\lambda_{k}|\mathcal{C}(E_{k})^{\frac{1}{p}}\|b_{k}\|_{L^{p',q}}\\
&\leq C(p,q)\|f\|_{\mathcal{M}^{p,\infty}}\sum_{k}|\lambda_{k}|,
\end{align*}
this gives $\mathcal{M}^{p,\infty}\hookrightarrow(\mathcal{B}^{p',q})^{\ast}$. Conversely, given $\mathcal{L}\in(\mathcal{B}^{p',q})^{\ast}$, we first notice that  
\begin{align*}
|\mathcal{L}(g)|\leq\|\mathcal{L}\|\mathcal{C}(E)^{\frac{1}{p}}\|g\|_{L^{p',q}},
\end{align*}
which leads to 
\begin{align}\label{4.2 prime}
\mathcal{L}(g)=\int_{\mathcal{R}}f(x)g(x)dx
\end{align}
holds for any $g\in L^{p',q}$ with $\{g\ne 0\}\subseteq E_{N}$. Subsequently, we can infer from (\ref{4.2 prime}) that $f$ is locally integrable on $\mathcal{R}$. Repeating the proof of \cite[Theorem 1.4.16 (v)]{GL}, one can show that the function $f$ satisfies 
\begin{align*}
\|f\chi_{K}\|_{L^{p,\infty}}\leq C(p,q)\|\mathcal{L}\|\mathcal{C}(K)^{\frac{1}{p}},
\end{align*}
which gives 
\begin{align*}
\|f\|_{\mathcal{M}^{p,\infty}}\leq C(p,q)\|\mathcal{L}\|.
\end{align*}
As before, (\ref{4.2 prime}) can be extended to all $g\in\mathcal{B}^{p',q}$, then  $(\mathcal{B}^{p',q})^{\ast}\hookrightarrow\mathcal{M}^{p,\infty}$. The proof is now complete.
\end{proof}

In the sequel, if $P(\cdot)$ is a property defined on $\mathcal{R}$, then we say that $P(\cdot)$ holds $\mathcal{C}$-quasi-everywhere provided that 
\begin{align*}
\mathcal{C}(\{x\in\mathcal{R}:P(x)~\text{is false}\})=0.
\end{align*}

\begin{lemma}\label{quasi-normed-everywhere}
Let $0<q\leq 1$. Assume that $\{f_{n}\}_{n=1}^{\infty}$ is a convergent sequence in $L^{1,q}(\mathcal{C})$, say, $f_{n}\rightarrow f$ in $L^{1,q}(\mathcal{C})$. Then there is a subsequence $\{f_{n_{k}}\}_{k=1}^{\infty}$ of $\{f_{n}\}_{n=1}^{\infty}$ converges to $f$ $\mathcal{C}$-quasi-everywhere. 
\end{lemma}

\begin{proof}
We first prove that $L^{1,q}(\mathcal{C})\hookrightarrow L^{1}(\mathcal{C})$. Fix a $0<t<\infty$. Then
\begin{align*}
t\mathcal{C}(\{|f|>t\})&=\left(q\int_{0}^{t}\mathcal{C}(\{|f|>t\})^{q}s^{q}\frac{ds}{s}\right)^{\frac{1}{q}}\\
&\leq\left(q\int_{0}^{t}\mathcal{C}(\{|f|>s\})^{q}s^{q}\frac{ds}{s}\right)^{\frac{1}{q}}\\
&\leq\left(q\int_{0}^{\infty}\mathcal{C}(\{|f|>s\})^{q}s^{q}\frac{ds}{s}\right)^{\frac{1}{q}}\\
&=q^{\frac{1}{q}}\|f\|_{L^{1,q}(\mathcal{C})}.
\end{align*}
Taking supremum over all $0<t<\infty$, we have 
\begin{align*}
\|f\|_{L^{1,\infty}(\mathcal{C})}\leq q^{\frac{1}{q}}\|f\|_{L^{1,q}(\mathcal{C})},
\end{align*}
and hence 
\begin{align}
\|f\|_{L^{1}(\mathcal{C})}&=\int_{0}^{\infty}\mathcal{C}(\{|f|>t\})dt\notag\\
&=\int_{0}^{\infty}[t\mathcal{C}(\{|f|>t\})]^{1-q+q}\frac{dt}{t}\notag\\
&\leq\|f\|_{L^{1,\infty}(\mathcal{C})}^{1-q}\int_{0}^{\infty}\mathcal{C}(\{|f|>t\})^{q}t^{q}\frac{dt}{t}\notag\\
&\leq q^{\frac{1-q}{q}}\|f\|_{L^{1,q}(\mathcal{C})}.\label{use embedding}
\end{align}
As a result, we obtain $f_{n}\rightarrow f$ in $L^{1}(\mathcal{C})$. For each $k=1,2,\ldots$, choose an $n_{k}\in\mathbb{N}$ such that 
\begin{align*}
\|f_{n_{k+1}}-f_{n_{k}}\|_{L^{1}(\mathcal{C})}<\frac{1}{2^{k}}
\end{align*}
with $n_{1}<n_{2}<\cdots$. Then
\begin{align*}
\sum_{k=1}^{\infty}\|f_{n_{k+1}}-f_{n_{k}}\|_{L^{1}(\mathcal{C})}\leq\sum_{k=1}^{\infty}\frac{1}{2^{k}}<\infty.
\end{align*}
Let
\begin{align*}
g=\sum_{k=1}^{\infty}|f_{n_{k+1}}-f_{n_{k}}|.
\end{align*}
Since $L^{1}(\mathcal{C})$ is normable, there exists a positive constant $\kappa=C(n,\alpha,s)$ such that 
\begin{align*}
\|g\|_{L^{1}(\mathcal{C})}=\sup_{n\in\mathbb{N}}\left\|\sum_{k=1}^{n}|f_{n_{k+1}}-f_{n_{k}}|\right\|_{L^{1}(\mathcal{C})}\leq\kappa\cdot\sup_{n\in\mathbb{N}}\sum_{k=1}^{n}\|f_{n_{k+1}}-f_{n_{k}}\|_{L^{1}(\mathcal{C})}<\infty.
\end{align*}
We deduce that $g\in L^{1}(\mathcal{C})$ and hence $|g(x)|<\infty$ $\mathcal{C}$-quasi-everywhere. This gives
\begin{align*}
\varphi(x)=\sum_{k=1}^{\infty}(f_{n_{k+1}}-f_{n_{k}})(x)=\lim_{k\rightarrow\infty}f_{n_{k}}(x)-f_{n_{1}}(x)
\end{align*}
exists as a real number $\mathcal{C}$-quasi-everywhere. Using Fatou's lemma, one has 
\begin{align*}
\|\varphi+f_{n_{1}}-f\|_{L^{1}(\mathcal{C})}\leq\liminf_{k\rightarrow\infty}\|f_{n_{k}}-f\|_{L^{1}(\mathcal{C})}=0.
\end{align*}
We conclude that $f(x)=\varphi(x)+f_{n_{1}}(x)=\lim\limits_{k\rightarrow\infty}f_{n_{k}}(x)$ $\mathcal{C}$-quasi-everywhere, which completes the proof.
\end{proof}

\begin{proof}[Proof of Theorem \ref{pre theorem}]
Denote by
\begin{align*}
M=\inf\{a>0: |\mu|^{\ast}(E)\leq a\cdot\mathcal{C}(E),~E\subseteq\mathcal{R}\}
\end{align*}
for simplicity. Assume at the moment that (\ref{1.2}) holds. Let $\mathcal{L}$ be a continuous linear functional on $\mathcal{L}^{1,q}(\mathcal{C})$. For any compact set $K\subseteq\mathcal{R}$, $\varphi\in C_{0}$, ${\rm supp}(\varphi)\subseteq K$, we have 
\begin{align*}
|\mathcal{L}(\varphi)|&\leq\|\mathcal{L}\|\|\varphi\|_{L^{1,q}(\mathcal{C})}\\
&\leq\|\mathcal{L}\|\left\|\left(\sup_{x\in\mathcal{R}}|\varphi(x)|\right)\chi_{K}\right\|_{L^{1,q}(\mathcal{C})}\\
&=\left(\frac{1}{q}\right)^{\frac{1}{q}}\mathcal{C}(K)\|\mathcal{L}\|\left(\sup_{x\in\mathcal{R}}|\varphi(x)|\right).
\end{align*}
We deduce that $\mathcal{L}$ is a measure on $\mathcal{R}$, say,
\begin{align*}
\mathcal{L}(\varphi)=\int_{\mathcal{R}}\varphi d\mu,\quad\varphi\in C_{0}.
\end{align*}
The total variation $|\mu|$ of $\mu$ satisfies 
\begin{align*}
|\mu|(\varphi)=\sup_{\substack{\psi\in C_{0}\\|\psi|\leq\varphi}}|\mu(\psi)|\leq\left(\frac{1}{q}\right)^{\frac{1}{q}}\mathcal{C}(K)\|\mathcal{L}\|\left(\sup_{x\in\mathcal{R}}|\varphi(x)|\right)
\end{align*}
for all $0\leq\varphi\in C_{0}$ with ${\rm supp}(\varphi)\subseteq K$. For any set $E\subseteq\mathcal{R}$, if $G\supseteq E$ is open, $\varphi\in C_{0}$, and $0\leq\varphi\leq\chi_{G}$, then
\begin{align*}
|\mu|(\varphi)\leq\left(\frac{1}{q}\right)^{\frac{1}{q}}\mathcal{C}(G)\|\mathcal{L}\|\left(\sup_{x\in\mathcal{R}}|\varphi(x)|\right)\leq\left(\frac{1}{q}\right)^{\frac{1}{q}}\mathcal{C}(G)\|\mathcal{L}\|.
\end{align*}
We obtain 
\begin{align*}
|\mu|(G)=\sup_{\substack{\varphi\in C_{0}\\0\leq\varphi\leq\chi_{G}}}|\mu(\varphi)|\leq\left(\frac{1}{q}\right)^{\frac{1}{q}}\mathcal{C}(G)\|\mathcal{L}\|,
\end{align*}
and hence 
\begin{align*}
|\mu|^{\ast}(E)=\inf_{\substack{E\subseteq G\subseteq\mathcal{R}\\ G~\text{open}}}|\mu|(G)\leq\left(\frac{1}{q}\right)^{\frac{1}{q}}\mathcal{C}(G)\|\mathcal{L}\|.
\end{align*}
Since $\mathcal{C}$ is outer regular (see \cite[Proposition 2.3.5]{AH}), this yields
\begin{align}\label{3.7}
|\mu|^{\ast}(E)\leq\left(\frac{1}{q}\right)^{\frac{1}{q}}\mathcal{C}(E)\|\mathcal{L}\|.
\end{align}
In particular, it holds that
\begin{align}\label{3.8}
\mathcal{C}(E)=0\quad\text{entails}\quad|\mu|^{\ast}(E)=0.
\end{align}
We also deduce by (\ref{3.7}) that
\begin{align}\label{second to last}
M\leq\|\mathcal{L}\|.
\end{align}
Combining (\ref{1.2}) with the above inequality, we obtain $(\mathcal{L}^{1,q})^{\ast}\hookrightarrow\mathfrak{M}$.

On the other hand, let $\mu$ be a measure on $\mathcal{R}$ that satisfying (\ref{1.2}) and denote by $\mathcal{L}_{\mu}$ the canonical linear functional that 
\begin{align*}
\mathcal{L}_{\mu}(\varphi)=\int_{\mathcal{R}}\varphi d\mu,\quad\varphi\in C_{0}. 
\end{align*}
If $\varphi\in C_{0}$, then 
\begin{align*}
|\mathcal{L}_{\mu}(\varphi)|&\leq\int_{\mathcal{R}}|\varphi|d|\mu|\\
&=\int_{0}^{\infty}\mu^{\ast}(\{|\varphi|>t\})dt\\
&\leq\int_{0}^{\infty}M\cdot\mathcal{C}(\{|\varphi|>t\})dt\\
&=M\|\varphi\|_{L^{1}(\mathcal{C})}\\
&\leq q^{\frac{1-q}{q}}M\|\varphi\|_{L^{1,q}(\mathcal{C})},
\end{align*}
where we have used (\ref{use embedding}) in the last inequality. Now we consider an arbitrary $f\in\mathcal{L}^{1,q}(\mathcal{C})$. There is a sequence $\{\varphi_{n}\}_{n=1}^{\infty}\subseteq C_{0}$ such that $\|\varphi_{n}-f\|_{L^{1,q}(\mathcal{C})}\rightarrow 0$ and $\varphi_{n}\rightarrow\varphi$ $\mathcal{C}$-quasi-everywhere (see Lemma \ref{quasi-normed-everywhere}). The above estimates show that
\begin{align*}
\int_{\mathcal{R}}|\varphi_{n}-\varphi_{m}|d|\mu|\leq q^{\frac{1-q}{q}}M\|\varphi_{n}-\varphi_{m}\|_{L^{1,q}(\mathcal{C})},\quad n,m\in\mathbb{N}.
\end{align*}
Hence $\varphi_{n}\rightarrow g$ in $L^{1}(|\mu|)$ for a $|\mu|$-integrable function $g$ on $\mathcal{R}$. We may assume that $\varphi_{n}\rightarrow g$ $\mu$-almost-everywhere. On the other hand, (\ref{3.8}) entails $\varphi_{n}\rightarrow f$ $|\mu|$-almost-everywhere, and hence $f=g$ $|\mu|$-almost-everywhere. This implies $f$ is $|\mu|$-measurable, $|\mu|$-integrable, and 
\begin{align*}
\left|\int_{\mathcal{R}}f d\mu\right|&=\left|\int_{\mathcal{R}}g d\mu\right|\\
&\leq\lim_{n\rightarrow\infty}\int_{\mathcal{R}}|\varphi_{n}| d|\mu|\\
&\leq q^{\frac{1-q}{q}}M\lim_{n\rightarrow\infty}\|\varphi_{n}\|_{L^{1,q}(\mathcal{C})}\\
&\leq q^{\frac{1-q}{q}}M\lim_{n\rightarrow\infty}\kappa_{0}(\|\varphi_{n}-f\|_{L^{1,q}(\mathcal{C})}+\|f\|_{L^{1,q}(\mathcal{C})})\\
&=q^{\frac{1-q}{q}}M\kappa_{0}\|f\|_{L^{1,q}(\mathcal{C})},
\end{align*}
where $\kappa_{0}=C(q)$. We deduce that
\begin{align*}
\mathcal{L}_{\mu}(\varphi)=\int_{\mathcal{R}}\varphi d\mu,\quad\varphi\in\mathcal{L}^{1,q}(\mathcal{C}),
\end{align*} 
and hence
\begin{align}\label{last}
\|\mathcal{L}_{\mu}\|\leq\kappa_{0}q^{\frac{1-q}{q}}M,
\end{align}
which shows that $\mathfrak{M}\hookrightarrow(\mathcal{L}^{1,q})^{\ast}$ in view of (\ref{1.2}) and hence $(\mathcal{L}^{1,q})^{\ast}\approx\mathfrak{M}$.

It remains to justify the claim (\ref{1.2}). Assuming that $M<\infty$ and $\varepsilon>0$, then there is an $0<a<M+\varepsilon$ such that $|\mu|^{\ast}(E)\leq a\cdot\mathcal{C}(E)$ for arbitrary set $E\subseteq\mathcal{R}$. For compact set $K\subseteq\mathcal{R}$ with $\mathcal{C}(K)>0$, we have both $\mathcal{C}(K)<\infty$ and $|\mu|^{\ast}(K)=|\mu|(K)<\infty$, whence 
\begin{align*}
\frac{|\mu|(K)}{\mathcal{C}(K)}\leq a\leq M+\varepsilon.
\end{align*}
Since $\varepsilon$ and $K$ are arbitrary, we obtain $\|\mu\|_{\mathfrak{M}}\leq M$. 

On the other hand, let $M>0$, $\|\mu\|_{\mathfrak{M}}<\infty$, and $\varphi\in C_{0}$. For fixed $0<t<\infty$, $\{|\varphi|\geq t\}\subseteq\mathcal{R}$ is compact. If $\mathcal{C}(\{|\varphi|\geq t\})=0$, then (\ref{3.8}) gives $|\mu|^{\ast}(\{|\varphi|\geq t\})=0$ and certainly
\begin{align*}
|\mu|^{\ast}(\{|\varphi|\geq t\})\leq\|\mu\|_{\mathfrak{M}}\cdot\mathcal{C}(\{|\varphi|\geq t\}).
\end{align*}
By the definition of $\|\mu\|_{\mathfrak{M}}$, the above inequality still holds when $\mathcal{C}(\{|\varphi|>t\})$ is nonzero. Then we claim that the canonical linear functional $\mathcal{L}_{\mu}$ induced by $\mu$ is continuous. If $\varphi\in C_{0}$, then
\begin{align*}
|\mathcal{L}_{\mu}(\varphi)|\leq\int_{0}^{\infty}\mu^{\ast}(\{|\varphi|\geq t\})dt\leq\|\mu\|_{\mathfrak{M}}\int_{0}^{\infty}\mathcal{C}(\{|\varphi|\geq t\})dt=\|\mu\|_{\mathfrak{M}}\|\varphi\|_{L^{1}(\mathcal{C})},
\end{align*}
For $\varphi\in\mathcal{L}^{1,q}(\mathcal{C})$, we repeat the argument right before (\ref{last}) with $\|\mu\|_{\mathfrak{M}}$ in place of $M$, then $\|\mathcal{L}_{\mu}\|\leq \kappa_{0}q^{\frac{1-q}{q}}\|\mu\|_{\mathfrak{M}}<\infty$, which shows the claim. Finally, using (\ref{second to last}), we obtain $M\leq\|\mathcal{L}_{\mu}\|<\infty$. Since $0<M<\infty$, given sufficiently small $\varepsilon>0$, the term $M-\varepsilon$ does not belong to the set $\{a>0: |\mu|^{\ast}(E)\leq a\cdot\mathcal{C}(E)\}$. Then there is an $E\subseteq\mathcal{R}$ with
\begin{align}\label{contra}
|\mu|^{\ast}(E)>(M-\varepsilon)\cdot\mathcal{C}(E). 
\end{align}
The set satisfies $\mathcal{C}(E)<\infty$ for otherwise it will contradict (\ref{contra}). Also note that (\ref{3.8}) gives $\mathcal{C}(E)>0$. Then the argument given in the proof of Lemma \ref{norm switching} gives 
\begin{align*}
M-\varepsilon\leq\frac{|\mu|^{\ast}(E)}{\mathcal{C}(E)}\leq\sup_{K}\frac{|\mu|(K)}{\mathcal{C}(K)}=\|\mu\|_{\mathfrak{M}}.
\end{align*}
The arbitrariness of $\varepsilon$ yields $M\leq\|\mu\|_{\mathfrak{M}}$. We conclude that (\ref{1.2}) holds and the proof is now complete.
\end{proof}

%\bigskip
%\noindent{\bf Data Availability} Not applicable.

%\bigskip
%\noindent{\bf Declarations}

%\bigskip
%\noindent{\bf Conflicts of interest} The corresponding author states that there is no conflict of interest.

%%%%%%%%%%%%%%%%%%%%%%%%%%%%%%%%%%
%%%%%%%%%%%%%%%%%%%%%%%%%%%%%%%%%%
%%%%%%%%%%%%%%%%
%              END
%%%%%%%%%%%%%%%%%%%%%%%%%%%%%%%%%%
%%%%%%%%%%%%%%%%%%%%%%%%%%%%%%%%%%
%%%%%%%%%%%%%%%%
\end{document}